%% file: StrictCModel.tex
\documentclass{amsart}
\usepackage{amsthm,amsfonts,amssymb,amsmath,amscd
 ,newclude
}
\usepackage{textcomp}
\usepackage{appendix}
\usepackage{xcolor}
\usepackage{longtable}
\usepackage{url}
\usepackage[utf8]{inputenc}
\usepackage[centerlast]{subfigure}
\usepackage[all,2cell, cmtip]{xy} \UseAllTwocells

\begin{document}

\title[Coherently commutative monoidal quasi-categories]{A homotopy theory of coherently commutative monoidal quasi-categories}
\author[A. Sharma]{Amit Sharma}
\email{asharm24@kent.edu}
\address {Department of mathematical sciences\\ Kent State University\\
  Kent,OH}

\date{May 2, 2020}
\input{my_definitions}

\begin{abstract}
 The main objective of this paper is to construct a symmetric monoidal closed model category of coherently commutative monoidal quasi-categories. We construct another model category structure whose fibrant objects are (essentially) those coCartesian fibrations which represent objects that are known as symmetric monoidal quasi-categories in the literature. We go on to establish a Quillen equivalence between the two model categories.
\end{abstract}

\maketitle

\tableofcontents

\include{Introduction}

\include{Setup}

\include{StrJQMdlStr}

\include{EinMdlStrQCat}
\include{QuillenEq}
\include{EinMdlMarQCat}

\include{CoCartFibModel}

 \appendix
 \include{EnrichmentOfGS}

\include{Cat-Localization}

\include{EinMdlStrSp-Nor}
 \bibliographystyle{amsalpha}
\bibliography{StrictCModel}

\end{document}

%% file: my_definitions.tex
%

\newcommand{\CONT}{\noindent}
\newcommand{\FIG}{Fig.\ }
\newcommand{\FIGS}{Figs.\ }
\newcommand{\SEC}{Sec.\ }
\newcommand{\SECS}{Secs.\ }
\newcommand{\TAB}{Table }
\newcommand{\TABS}{Tables }
\newcommand{\EQ}{Eq.\ }
\newcommand{\EQS}{Eqs.\ }
\newcommand{\APP}{Appendix }
\newcommand{\APPS}{Appendices }
\newcommand{\CHP}{Chapter }
\newcommand{\CHPS}{Chapters }

\newcommand{\OFF}{\emph{G2off}~}
\newcommand{\TOO}{\emph{G2Too}~}
\newcommand{\CatS}{Cat_{\bigS}}
\newcommand{\PicS}{{\underline{\pic}}^{\oplus}}
\newcommand{\HPicS}{{Hom^{\oplus}_{\pic}}}

\newtheorem{thm}{Theorem}[section]
\newtheorem{lem}[thm]{Lemma}
\newtheorem{conj}[thm]{Conjecture}
\newtheorem{coro}[thm]{Corollary}
\newtheorem{prop}[thm]{Proposition}

\theoremstyle{definition}
\newtheorem{df}[thm]{Definition}
\newtheorem{nota}[thm]{Notation}

\newtheorem{ex}[thm]{Example}
\newtheorem{exs}[thm]{Examples}

\theoremstyle{remark}
\newtheorem*{note}{Note}
\newtheorem{rem}{Remark}
\newtheorem{ack}{Acknowledgments}

\newcommand{\ChI}{{\textit{\v C}}\textit{ech}}
\newcommand{\Ch}{{\v C}ech}

\newcommand{\ChZG}{hermitian line $0$-gerbe}
\renewcommand{\theack}{$\! \! \!$}

\newcommand{\ChG}{flat hermitian line $1$-gerbe}
\newcommand{\ChC}{hermitian line $1$-cocycle}
\newcommand{\ChGG}{flat hermitian line $2$-gerbe}
\newcommand{\ChCC}{hermitian line $2$-cocycle}
\newcommand{\id}{id}
\newcommand{\LC}{\mathfrak{C}}
\newcommand{\Coker}{Coker}
\newcommand{\Com}{Com}
\newcommand{\Hom}{Hom}
\newcommand{\Mor}{Mor}
\newcommand{\Map}{Map}
\newcommand{\alg}{alg}
\newcommand{\an}{an}
\newcommand{\Ker}{Ker}
\newcommand{\Ob}{Ob}
\newcommand{\Proj}{\mathbf{Proj}}
\newcommand{\topo}{\mathbf{Top}}
\newcommand{\kan}{\mathcal{K}}
\newcommand{\pkan}{\mathcal{K}_\bullet}
\newcommand{\Kan}{\mathbf{Kan}}
\newcommand{\pKan}{\mathbf{Kan}_\bullet}
\newcommand{\QCat}{\mathbf{QCat}}
\newcommand{\gp}{\mathcal{A}_\infty}
\newcommand{\mdl}{\mathcal{M}\textit{odel}}
\newcommand{\sSets}{\mathbf{\S}}
\newcommand{\Sets}{\mathbf{Sets}}
\newcommand{\sSetsM}{\mathbf{\S}^+}
\newcommand{\sSetsQ}{(\mathbf{\sSets, Q})}
\newcommand{\sSetsMQ}{(\mathbf{\sSetsM, Q})}
\newcommand{\sSetsK}{(\mathbf{\sSets, \Kan})}
\newcommand{\pSSets}{\mathbf{\sSets}_\bullet}
\newcommand{\pSSetsK}{(\mathbf{\sSets}_\bullet, \Kan)}
\newcommand{\pSSetsQ}{(\mathbf{\sSets_\bullet, Q})}
\newcommand{\cyl}{\mathbf{Cyl}}
\newcommand{\lin}{\mathcal{L}_\infty}
\newcommand{\Vect}{\mathbf{Vect}}
\newcommand{\Aut}{Aut}
\newcommand{\pic}{\mathcal{P}\textit{ic}}
\newcommand{\Dlin}{\pic}
\newcommand{\bigS}{\mathbf{S}}
\newcommand{\bigA}{\mathbf{A}}
\newcommand{\bhom}{\mathbf{hom}}
\newcommand{\bhomK}{\mathbf{hom}({\textit{K}}^+,\textit{-})}
\newcommand{\Bhom}{\mathbf{Hom}}
\newcommand{\bhomk}{\mathbf{hom}^{{\textit{k}}^+}}
\newcommand{\Dlino}{\pic^{\textit{op}}}
\newcommand{\lino}{\mathcal{L}^{\textit{op}}_\infty}
\newcommand{\lind}{\mathcal{L}^\delta_\infty}
\newcommand{\linK}{\mathcal{L}_\infty(\kan)}
\newcommand{\linC}{\mathcal{L}_\infty\text{-category}}
\newcommand{\linCs}{\mathcal{L}_\infty\text{-categories}}
\newcommand{\ainCs}{\text{additive} \ \infty-\text{categories}}
\newcommand{\ainC}{\text{additive} \ \infty-\text{category}}
\newcommand{\inC}{\infty\text{-category}}
\newcommand{\inCs}{\infty\text{-categories}}
\newcommand{\gS}{{\Gamma}\text{-space}}
\newcommand{\gSet}{{\Gamma}\text{-set}}
\newcommand{\ggS}{\Gamma \times \Gamma\text{-space}}
\newcommand{\gSs}{\Gamma\text{-spaces}}
\newcommand{\gSets}{\Gamma\text{-sets}}
\newcommand{\ggSs}{\Gamma \times \Gamma\text{-spaces}}
\newcommand{\gO}{\Gamma-\text{object}}
\newcommand{\gSCat}{{\Gamma}\text{-space category}}
\newcommand{\pss}{\mathbf{S}_\bullet}
\newcommand{\gSC}{{{{\Gamma}}\mathcal{S}}}
\newcommand{\gSCM}{{{{\Gamma}}\mathcal{S}^+}}
\newcommand{\pGSC}{{{{\Gamma}}\mathcal{S}}_\bullet}
\newcommand{\pGSCStr}{{{{\Gamma}}\mathcal{S}}_\bullet^{\textit{str}}}
\newcommand{\ggSC}{{\Gamma\Gamma\mathcal{S}}}
\newcommand{\gSD}{\mathbf{D}(\gSC^{\textit{f}})}
\newcommand{\sCat}{\mathbf{sCat}}
\newcommand{\pSCat}{\mathbf{sCat}_\bullet}
\newcommand{\gSetCat}{{{{\Gamma}}\mathcal{S}\textit{et}}}
\newcommand{\Dhom}{\mathbf{R}Hom_{\pic}}
\newcommand{\gop}{\Gamma^{\textit{op}}}
\newcommand{\fU}{\mathbf{U}}
\newcommand{\cDN}{\underset{\mathbf{D}[\textit{n}^+]}{\circ}}
\newcommand{\cDK}{\underset{\mathbf{D}[\textit{k}^+]}{\circ}}
\newcommand{\cDL}{\underset{\mathbf{D}[\textit{l}^+]}{\circ}}
\newcommand{\cD}{\underset{\gSD}{\circ}}
\newcommand{\cDT}{\underset{\gSD}{\widetilde{\circ}}}
\newcommand{\ppsSets}{\sSets_{\bullet, \bullet}}
\newcommand{\gdHom}{\underline{Hom}_{\gSD}}
\newcommand{\HomU}{\underline{Hom}}
\newcommand{\ominf}{\Omega_\infty}
\newcommand{\ev}{ev}
\newcommand{\cu}{C(X;\mathfrak{U}_I)}
\newcommand{\Sing}{Sing}
\newcommand{\AlgEin}{\A\textit{lg}_{\E_\infty}}
\newcommand{\SFunc}[2]{\mathbf{SFunc}({#1} ; {#2})}
\newcommand{\unit}[1]{\mathrm{1}_{#1}}
\newcommand{\liminj}{\varinjlim}
\newcommand{\limproj}{\varprojlim}
\newcommand{\HMapC}[3]{\mathcal{M}\textit{ap}^{\textit{h}}_{#3}(#1, #2)}
\newcommand{\tensPGSR}[2]{#1 \underset{\gSR}\wedge #2}
\newcommand{\pTensP}[3]{#1 \underset{#3}\wedge #2}
\newcommand{\MGCat}[2]{\underline{\map}_{\gSC}({#1},{ #2})}
\newcommand{\MGBoxCat}[2]{\underline{\map}_{\gSC}^{\Box}({#1},{ #2})}
\newcommand{\TensPFunc}[1]{- \underset{#1} \otimes -}
\newcommand{\TensP}[3]{#1 \underset{#3}\otimes #2}
\newcommand{\MapC}[3]{\mathcal{M}\textit{ap}_{#3}(#1, #2)}
\newcommand{\bHom}[3]{{#2}^{#1}}
\newcommand{\gn}[1]{\Gamma^{#1}}
\newcommand{\gnk}[2]{\Gamma^{#1}({#2}^+)}
\newcommand{\gnf}[2]{\Gamma^{#1}({#2})}
\newcommand{\ggn}[1]{\Gamma\Gamma^{#1}}
\newcommand{\Nat}{\mathbb{N}}
\newcommand{\partition}[2]{\delta^{#1}_{#2}}
\newcommand{\inclusion}[2]{\iota^{#1}_{#2}}
\newcommand{\EinQC}{\text{coherently commutative monoidal quasi-category}} 
\newcommand{\EinQCs}{\text{coherently commutative monoidal quasi-categories}}
\newcommand{\pHomCat}[2]{[#1,#2]_{\bullet}}
\newcommand{\CatHom}[3]{[#1,#2]^{#3}}
\newcommand{\pCatHom}[3]{[#1,#2]_\bullet^{#3}}
\newcommand{\EinC}{\text{coherently commutative monoidal category}}
\newcommand{\EinCs}{\text{coherently commutative monoidal categories}}
\newcommand{\EinLO}{E_\infty{\text{- local object}}}
\newcommand{\EinSLO}{\E_\infty\S{\text{- local object}}}
\newcommand{\Ein}{E_\infty}
\newcommand{\EinS}{E_\infty{\text{- space}}}
\newcommand{\EinSs}{E_\infty{\text{- spaces}}}
\newcommand{\PCat}{\mathbf{Perm}}
\newcommand{\nor}[1]{{#1}^\textit{nor}}
\newcommand{\pSSetsHom}[3]{[#1,#2]_\bullet^{#3}}
\newcommand{\PNat}{\overline{\L}}
\newcommand{\PStr}{\L}
\newcommand{\Gn}[1]{\Gamma[#1]}
\newcommand{\GIH}{\Gamma\textit{H}_{\textit{in}}}
\newcommand{\QStr}[1]{\L_\bullet(\ud{#1})}
\newcommand{\QStrF}{\L_\bullet}
\newcommand{\Kbar}{\overline{\K}}
\newcommand{\gPerm}{{\Gamma\PCat}}
\newcommand{\gCat}{{\Gamma\Cat}}
\newcommand{\MapS}[3]{\map_{#3}(#1, #2)}
\newcommand{\sSetsMG}{\sSetsM / N(\gop)}
\newcommand{\sSetsMGen}[1]{\sSetsM / N(#1)}
\newcommand{\pF}{\mathfrak{F}_\bullet^+(\gop)}
\newcommand{\pN}{{\textit{N}}_\bullet^+(\gop)}
\newcommand{\pFX}[1]{\mathfrak{F}_{#1}^+(\gop)}
\newcommand{\pNX}[1]{{\textit{N}}_{#1}^+(\gop)}
\newcommand{\nGop}{N(\gop)}
\newcommand{\sSetsMGSM}{(\sSetsM/ N(\gop), \otimes)}
\newcommand{\sSetsGen}[1]{\sSets/ #1}
\newcommand{\ovCatGen}[2]{#1/ #2}
\newcommand{\coMdl}[1]{(\sSets/ #1, \mathbf{L})}
\newcommand{\ccMdl}[1]{(\sSetsM/ #1, \mathbf{cC})}
\newcommand{\NElG}[1]{N(\ovCatGen{{#1}^+}{\gop})}

\def\Pic{\mathbf{2}\mathcal P\textit{ic}}
\def\nc{\mathbb C}

\def\Z{\mathbb Z}
\def\P{\mathbb P}
\def\J{\mathcal J}
\def\I{\mathcal I}
\def\nC{\mathbb C}
\def\H{\mathcal H}
\def\A{\mathcal A}
\def\C{\mathcal C}
\def\D{\mathcal D}
\def\E{\mathcal E}
\def\G{\mathcal G}
\def\B{\mathcal B}
\def\L{\mathcal L}
\def\U{\mathcal U}
\def\K{\mathcal K}
\def\El{\mathcal E{\textit{l}}}

\def\M{\mathcal M}
\def\O{\mathcal O}
\def\R{\mathcal R}
\def\S{\mathcal S}
\def\N{\mathcal N}

\newcommand{\undertilde}[1]{\underset{\sim}{#1}}
\newcommand{\abs}[1]{{\lvert#1\rvert}}
\newcommand{\mC}[1]{\mathfrak{C}(#1)}
\newcommand{\sigInf}[1]{\Sigma^{\infty}{#1}}
\newcommand{\x}[4]{\underset{#1, #2}{ \overset{#3, #4} \prod }}
\newcommand{\mA}[2]{\textit{Add}^n_{#1, #2}}
\newcommand{\mAK}[2]{\textit{Add}^k_{#1, #2}}
\newcommand{\mAL}[2]{\textit{Add}^l_{#1, #2}}
\newcommand{\Mdl}[2]{\L_\infty}
\newcommand{\inv}[1]{#1^{-1}}
\newcommand{\Lan}[2]{\mathbf{Lan}_{#1}(#2)}

\newcommand{\del}{\partial}
\newcommand{\sCatO}{\mathcal{S}Cat_\O}
\newcommand{\FCgop}{\mathbf{F}\mC{N(\gop)}}
\newcommand{\hProd}{{\overset{h} \oplus}}
\newcommand{\hProdn}{\underset{n}{\overset{h} \oplus}}
\newcommand{\hProdk}[1]{\underset{#1}{\overset{h} \oplus}}
\newcommand{\map}{\mathcal{M}\textit{ap}}
\newcommand{\SMGS}[2]{\map_{\gSC}({#1},{ #2})}
\newcommand{\MGS}[2]{\underline{\map}_{\gSC}({#1},{ #2})}
\newcommand{\MGSBox}[2]{\underline{\map}^{\Box}_{\gSC}({#1},{ #2})}
\newcommand{\Aqcat}[1]{\underline{#1}^\oplus}
\newcommand{\Cat}{\mathbf{Cat}}
\newcommand{\Sp}{\mathbf{Sp}}
\newcommand{\SpStb}{\mathbf{Sp}^{\textit{stable}}}
\newcommand{\SpStr}{\mathbf{Sp}^{\textit{strict}}}
\newcommand{\Sspec}{\mathbb{S}}
\newcommand{\ud}[1]{\underline{#1}}
\newcommand{\inrt}{\mathbf{Inrt}}
\newcommand{\act}{\mathbf{Act}}
\newcommand{\StrSMHom}[2]{[#1,#2]_\otimes^{\textit{str}}}
\newcommand{\Sh}[1]{{#1}^\sharp}
\newcommand{\Fl}[1]{{#1}^\flat}
\newcommand{\Nt}[1]{{#1}^\natural}
\newcommand{\Flmap}[3]{\Fl{\left[#1, #2\right]}_{#3}}
\newcommand{\Shmap}[3]{\Sh{\left[#1, #2 \right]}_{#3}}
\newcommand{\mRN}[1]{N^+_{#1}(\gop)}
\newcommand{\mRNGen}[2]{\int_{+}^{d \in {#2}}{#1}}
\newcommand{\mRNL}[1]{{\mathfrak{F}}^+_{#1}(\gop)}
\newcommand{\mRNR}{\mRNGenR{\gop}}
\newcommand{\mRNGenL}[2]{{\mathfrak{F}}^+_{#1}(#2)}
\newcommand{\mapG}[2]{[#1, #2]_{\gop}^+}
\newcommand{\mapFl}[2]{\Fl{[#1, #2]}}
\newcommand{\mapSh}[2]{\Sh{[#1, #2]}}
\newcommand{\mapMS}[2]{[#1, #2]^+}
\newcommand{\ExpG}[2]{{\left({#2}\right)}^{[#1]}}
\newcommand{\expG}[2]{{{#2}}^{[#1]}}
\newcommand{\mapGen}[3]{[#1, #2]_{#3}}
\newcommand{\mapGenM}[3]{[#1, #2]^+_{#3}}
\newcommand{\rNGen}[2]{\int^{d \in {#2}}{#1}}
\newcommand{\rNGenL}[2]{\mathfrak{L}_{#2}(#1)}
\newcommand{\rN}[1]{\int^{n^+ \in \gop} {#1}}
\newcommand{\rNGenR}[1]{\mathfrak{R}_{#1}}
\newcommand{\mRNGenR}[1]{\mathfrak{R}^{+}_{#1}}
\newcommand{\cn}[1]{\mathbf{id}(d)_{{#1}}}
\newcommand{\gCLM}[1]{\mathfrak{L}^+_{#1}}

%% file: Introduction.tex
\section[Introduction]{Introduction}
A \emph{symmetric monoidal category} is a category equipped with a multiplicative structure which is associative, unital and commutative only upto natural (coherence) isomorphisms. A \emph{quasi-category} is a simplicial set which satisfies the weak Kan condition namely every inner horn has a filler.  In this paper we study quasi-categories which are equipped with a \emph{coherently commutative} multiplicative structure and thereby generalize the notion of symmetric monoidal categories to higher categories. Such quasi-categories most commonly arise as (simplicial) nerves of simplicial model categories which are equipped a compatible symmetric monoidal structure see \cite{NS}. These quasi-categories played a prominent role in Jacob Lurie's work on the \emph{cobordism hypothesis}. Every \emph{stable} quasi-category \cite{JL4} is equipped with a multiplicative structure which is coherently commutative. The \emph{coherence theorem} for symmetric monoidal categories states that the category of (small) symmetric monoidal categories is equivalent to the category of algebras over the \emph{categorical Barrat-Eccles operad} in $\Cat$. We recall that the categorical Barrat-Eccles operad is an $E_\infty$-operad in $\Cat$.
In a subsequent paper we will prove a similar theorem for the quasi-categories equipped with a coherently commutative multiplicative structure. There are several different models present in the literature which were developed to encode a coherently commutative multiplicative structure on simplicial sets. The most commonly used model is based on operads.
An $E_\infty$-simplicial set is a simplicial set equipped with a \emph{coherently commutative multiplicative} structure which is encoded by an action of an $E_\infty$-operad. In other words an $E_\infty$-simplicial set is an \emph{algebra} over an $E_\infty$-operad in the category of simplicial sets $\sSets$. There are two model category structures on the category $\sSets$ namely the \emph{standard} or the \emph{Kan} model category structure and the \emph{Joyal} model category structure which is also referred to as the model category structure of \emph{quasi-categories}. In this paper we will only be working with the later model category structure. The category of $E_\infty$-simplicial sets inherits a model category structure from the Joyal model category structure, see \cite{BM2}. A fibrant object in this model category can be described as a quasi-category equipped with a coherently commutative multiplicative structure which is encoded by an action of an $E_\infty$-operad. However this model category is NOT symmetric monoidal closed. The main objective of this paper is to overcome this shortcoming by presenting a new model for coherently commutative monoidal quasi-categories based on $\gSs$. Another model to encode a coherently commutative multiplicative structure on simplicial sets was presented by Jacob Lurie  in his book \cite{JL2} which he called \emph{symmetric monoidal quasi-categories}.
He modelled these objects as \emph{co-cartesian fibrations} over a quasi-category which is the nerve of a skeletal category of based finite sets $\gop$ whose objects are $n^+ = \lbrace 0, 1, 2, \dots, n \rbrace$. In this paper we take a dual perspective namely we model  these as functors from $\gop$ into $\sSetsQ$. However Lurie does not construct a model category structure on his symmetric monoidal quasi-categories. Yet another model to encode a coherently commutative multiplicative structure on simplicial sets was presented by Kodjabachev and Sagave in the paper \cite{KS}. The authors present a rigidification of an $E_\infty$-quasi-category by replacing it with a commutative monoid in a symmetric monoidal functor category. They go on to construct a zig-zag of Quillen equivalences between a suitably defined model category structure on the category of commutative monoids mentioned above and a model category of $E_\infty$-simplicial sets. However they were unable to show the existence of a symmetric monoidal closed model category structure.

 A $\gS$ is a functor from the category $\gop$ into the category of simplicial sets $\sSets$. The category $\gSC$ of $\gSs$ is the category of functors and natural transformations $[\gop, \sSets]$.
 A \emph{normalized} $\gS$ is a functor $X:\gop \to \pSSets$ such that $X(0^+) = \ast$. The category of normalized $\gSs$ $\pGSC$ is the full subcategory of the functor category $[\gop; \pSSets]$ whose objects are normalized $\gSs$. In the paper \cite{segal} Segal introduced a notion of normalized $\gSs$ and showed that they give rise to a homotopy category which is equivalent to the homotopy category of connective spectra. Segal's $\gSs$ were renamed \emph{special} $\gSs$ by Bousfield and Friedlander in \cite{BF78} who constructed a model category structure on the category of all normalized $\gSs$ $\pGSC$. The two authors go on to prove that the homotopy category obtained by inverting \emph{stable weak equivalences} in $\pGSC$ is equivalent to the homotopy category of connective spectra.
In the paper \cite{schwede} Schwede constructed a symmetric monoidal closed model category structure on the category of \emph{normalized} $\gSs$ which he called the \emph{stable} Q-model category. The fibrant objects in this model category can be described as \emph{coherently commutative group objects} in the category of (pointed) simplicial sets $\pSSets$, where the latter category is endowed with the \emph{Kan} model category structure. 
 The objective of Schwede's construction was to establish normalized $\gSs$ as a model for \emph{connective spectra}. In this paper we extend the ideas in \cite{schwede} to study coherently commutative monoidal objects in the model category of quasi-categories and thereby generalizing the theory of \emph{symmetric monoidal categories}.  We construct a new symmetric monoidal closed model category structure on the category of $\gSs$ $\gSC$. Our model category is constructed along the lines of Schwede's construction and we call it the \emph{JQ}-model category structure. The fibrant objects in our model category structure can be described as \emph{coherently commutative monoidal quasi-categories}. We will show that the \emph{JQ}-model category is symmetric monoidal closed under the \emph{Day convolution product}.

The category of (pointed) simplicial sets $\pSSets$ inherits a model category structure from the Joyal model category. This model category is symmetric monoidal closed under the \emph{smash product of (pointed) simplicial sets}, see \cite{JT2}. 
We construct a new model category structure on the category of nomalized $\gSs$ $\pGSC$.  The fibrant object of this model category can be described as \emph{strictly unital coherently commutative monoidal quasi-categories}. We will refer to this model category as the \emph{normalized} JQ-model category. In the paper \cite{lydakis} Lydakis constructed a \emph{smash product of $\gSs$} and showed that it endows $\pGSC$ with a closed symmetric monoidal structure. We will show that the normalized $JQ$-model category structure is compatible with the smash product of $\gSs$ \emph{i.e.}  the normalized \emph{JQ} model category is symmetric monoidal closed under the smash product.   Another significant result of this paper is that the obvious forgetful functor $U:\pGSC \to \gSC$ is the right Quillen functor of a Quillen equivalence between the JQ-model category and the normalized JQ-model category. This result is indicative of the presence of a \emph{weak semi-additive} structure in the JQ-model category.

A notion of a symmetric monoidal quasi-category  based on coCartesian fibrations of simplicial sets was introduced in \cite{JL2}. A cocartesian fibration $p:X \to \nGop$ determines a fibrant object $(\Nt{X}, p)$ in a \emph{coCartesian model category structure} on the overcategory $\sSetsMG$ , by \emph{marking} its \emph{coCartesian edges}. We construct another model category structure on $\sSetsMG$, denoted $\sSetsMGSM$, by localizing the coCartesian model category structure. An object $(X, p)$ in $\sSetsMGSM$ is fibrant if and only if it is isomorphic to an object $(\Nt{Y}, q)$ determined by a coCartesian fibration $q:Y \to \nGop$ which is a symmetric monoidal quasi-category. We go on further to show that the Quillen pair $(\mRNL{\bullet}, \mRN{\bullet})$ from \cite[Prop. 3.2.5.18]{JL} is a Quillen equivalence between the model category of coherently commutative monoidal quasi-categories and $\sSetsMGSM$. The right Quillen functor $\mRN{\bullet}$ can be viewed as a higher categorical version of the Grothendieck construction functor.

\begin{ack}
 The author is thankful to Andre Joyal for helping him understand the theory of quasi-categories and also for providing a detailed write up of appendix \ref{Cat-Local} to the author.
\end{ack}

%% file: Setup.tex
\section{The Setup}
In this section we will collect the machinery needed for various constructions in this paper.

\input{MarkSS}
\input{RevGCat}

\input{cCMdlStr}

%% file: MarkSS.tex
\section[Marked simplicial sets]{A review of marked simplicial sets}
\label{mar-sSets}
In this appendix we will review the theory of marked simplicial sets. Later in this paper we will develop a theory of coherently commutative monoidal objects in the category of marked simplicial sets.

\begin{df}
	\label{mar-sSet}
	A \emph{marked} simplicial set is a pair $(X, \E)$, where $X$ is a simplicial set and $\E$ is a set of edges of $X$ which contains every degenerate edge of $X$. We will say that an edge of $X$ is \emph{marked} if it belongs to $\E$.
	A morphism $f:(X, \E) \to (X', \E')$ of marked simplicial sets is a simplicial map $f:X \to X'$ having the property that $f(\E) \subseteq \E'$. We denote the category of marked simplicial sets by $\sSetsM$.
	\end{df}

Every simplicial set $S$ may be regarded as a marked simplicial set in many ways. We mention two extreme cases: We let $\Sh{S} = (S, S_1)$ denote the marked simplicial set in which every edge is marked. We denote by $\Fl{S} = (S, s_0(S_0))$ denote the marked simplicial set in which only the degerate edges of $S$ have been marked.

The category $\sSetsM$ is \emph{cartesian-closed}, \emph{i.e.} for each pair of objects $X, Y \in Ob(\sSetsM)$, there is an internal mapping object $[X, Y]^+$ equipped with an \emph{evaluation map} $[X, Y]^+ \times X \to Y$ which induces a bijection:
\[
\sSetsM(Z, [X, Y]^+) \overset{\cong} \to \sSetsM(Z \times X, Y),
\]
for every $Z \in \sSetsM$.
\begin{nota}
	We denote by $\Fl{[X, Y]}$ the underlying simplicial set of $[X, Y]^+$.
	\end{nota}
The mapping space $\Fl{[X, Y]}$ is charaterized by the following bijection:
\[
\sSets(K, \Fl{[X, Y]}) \overset{\cong} \to \sSetsM(\Fl{K} \times X, Y),
\]
for each simplicial set $K$.
\begin{nota}
	We denote by $\Sh{[X, Y]}$ the simplicial subset of $\Fl{[X, Y]}$ consisting of all simplices $\sigma \in \Fl{[X, Y]}$ such that every edge of $\sigma$ is a marked edge of $[X, Y]^+$.
\end{nota}
The mapping space $\Sh{[X, Y]}$ is charaterized by the following bijection:
\[
\sSets(K, \Sh{[X, Y]}) \overset{\cong} \to \sSetsM(\Sh{K} \times X, Y),
\]
for each simplicial set $K$.

 The Joyal model category structure on $\sSets$ has the following analog for marked simplicial sets:
 \begin{thm}
 	\label{Joyal-sSetsM}
 	There is a left-proper, combinatorial model category structure on the category of marked simplicial sets $\sSetsM$ in which a morphism $p:X \to Y$ is a
 	\begin{enumerate}
 		\item cofibration if the simplicial map between the underlying simplicial sets is a cofibration in $\sSetsQ$, namely a monomorphism.
 		
 		\item a weak-equivalence if the induced simplicial map on the mapping spaces
 		\[
 		\Fl{[p, \Nt{K}]}:\Fl{[X, \Nt{K}]} \to \Fl{[Y, \Nt{K}]}
 		\]
 		is a weak-categorical equivalence, for each quasi-category $K$.
 		
 		\item fibration if it has the right lifting property with respect to all maps in $\sSetsM$ which are simultaneously cofibrations and weak equivalences.
 		
 		\end{enumerate}
 	Further, the above model category structure is enriched over the Joyal model category, i.e. it is a $\sSetsQ$-model category.
 	\end{thm}
 The above theorem follows from \cite[Prop. 3.1.3.7]{JL}.
 \begin{nota}
 	We will denote the model category structure in Theorem \ref{Joyal-sSetsM} by $\sSetsMQ$ and refer to it either as the \emph{Joyal} model category of \emph{marked} simplicial sets or as the model category of marked quasi-categories.
 \end{nota}
 \begin{thm}
 	\label{Cart-cl-Mdl-S-plus}
 	The model category $\sSetsMQ$ is a cartesian closed model category.
 \end{thm}
 \begin{proof}
 	The theorem follows from \cite[Corollary 3.1.4.3]{JL} by taking $S = T = \Delta[0]$.
 \end{proof}
 There is an obvious forgetful functor $U:\sSetsM \to \sSets$. This forgetful functor has a left adjoint $\Fl{(-)}:\sSets \to \sSetsM$.
 \begin{thm}
 	\label{Quil-eq-QCat-MQCat}
 	The adjoint pair of functors $(\Fl{(-)}, U)$ determine a Quillen equivalence between the Joyal model category of marked simplicial sets and the Joyal model category of simplicial sets.
 	\end{thm}
 The proof of the above theorem follows from \cite[Prop. 3.1.5.3]{JL}.

%% file: RevGCat.tex
\subsection[Review of $\gSs$]{Review of $\gSs$}
\label{Rev-gamma-cats}
In this subsection we will briefly review the theory of $\gSs$. We begin by introducing some notations which will be used throughout the paper.
\begin{nota}
We will denote by $\ud{n}$ the finite set $\lbrace 1, 2, \dots, n \rbrace$ and by $n^+$ the based set $\lbrace 0, 1, 2, \dots, n \rbrace$ whose basepoint is the element $0$.
\end{nota}
\begin{nota}
 We will denote by $\N$ the skeletal category of finite unbased sets whose objects are $\ud{n}$ for all $n \ge 0$ and maps are functions of unbased sets. The category $\N$ is a (strict) symmetric monoidal category whose symmetric monoidal structure will be denoted by $+$. For to objects $\ud{k}, \ud{l} \in \N$ their \emph{tensor product} is defined as follows:
 \[
 \ud{k} + \ud{l} := \ud{k + l}.
 \]
\end{nota}
\begin{nota}
 We will denote by $\gop$ the skeletal category of finite based sets whose objects are $n^+$ for all $n \ge 0$ and maps are functions of based sets.
\end{nota}

\begin{nota}
 We denote by $\inrt$ the subcategory of $\gop$ having the same set of objects as $\gop$
 and intert morphisms.
\end{nota}
\begin{nota}
 We denote by $\act$ the subcategory of $\gop$ having the same set of objects as $\gop$
 and active morphisms.
\end{nota}
\begin{nota}
A map $f:\ud{n} \to \ud{m}$ in the category $\N$ uniquely determines an active map in $\gop$ which we will denote by $f^+:n^+ \to m^+$.
This map agrees with $f$ on non-zero elements of $n^+$.
\end{nota}
\begin{nota}
 Given a morphism $f:n^+ \to m^+$ in $\gop$, we denote by $\text{Supp}(f)$ the largest
 subset of $\N$ whose image under $f$ does not caontain the basepoint of $m^+$.
 The set $\text{Supp}(f)$ inherits an order from $\N$ and therefore could be regarded as
 an object of $\N$. We denote by $\text{Supp}(f)^+$ the based set $\text{Supp}(f) \sqcup \lbrace 0 \rbrace$
 regarded as an object of $\gop$ with order inherited from $\N$.
\end{nota}

\begin{prop}
 Each morphism in $\gop$ can be uniquely factored into a composite of an inert map followed
 by an active map in $\gop$.
\end{prop}

\begin{proof}
Any map $f:n^+ \to m^+$ in the category $\gop$ can be factored as follows:
\begin{equation}
\label{fact-in-gamma-op}
  \xymatrix@C=11mm{
 n^+  \ar[rd]_{f_{\textit{inrt}}}   \ar[rr]^{f}  &&m^+   \\
  &\text{Supp}(f)^+  \ar[ru]_{f_{\textit{act}}}
 }
\end{equation}
where $\text{Supp}(f) \subseteq n$ is the \emph{support} of the function $f$
\emph{i.e.}  $\text{Supp}(f)$ is the largest subset of $n$ whose elements are mapped
by $f$ to a non zero element of $m^+$. The map $f_{inrt}$ is the projection of
$n^+$ onto the support of $f$ and therefore $f_{inrt}$ is an inert map. The map
$f_{act}$ is the restriction of $f$ to $\text{Supp}(f) \subset \N$,
therefore it is an \emph{active} map in $\gop$.

\end{proof}

\begin{df}
	\label{gamma-space}
	A $\gS$ is a functor from $\gop$ into the category of simplicial sets $\sSets$.
\end{df}
\begin{df}
	\label{nor-gamma-space}
	A \emph{normalized} $\gS$ is $X$ a $\gS$ which satisfies the normalization condition namely $X(0^+) \cong \ast$.
\end{df}

In this paper we will also study functors from $\gop$ into the category of marked simplicial sets:
\begin{df}
	\label{Mar-GS}
	A \emph{marked} $\gS$ is a functor from $\gop$ to the category of marked simplicial sets $\sSetsM$. A morphism of marked $\gSs$ is a natural transformation between marked $\gSs$.
	\end{df}

\begin{nota}
	We denote the category of all marked $\gSs$ and morphisms of $\gSs$ by $\gSCM$.
	\end{nota}

The adjoint pair $(\Fl{(-)}, U)$ defined above induces an adjunction
\begin{equation}
\label{adj-GS-MGS}
\Fl{\Gamma(-)}:\gSC \rightleftharpoons \gSCM :U
\end{equation}
where the left adjoint $\Fl{\Gamma(-)}$ is the induced map $[\gop, \Fl{(-)}]:\gSC = [\gop, \sSets] \to [\gop, \sSetsM] = \gSCM$.

%% file: cCMdlStr.tex
\subsection{Review of the coCartesian model structure}
\label{coCart-mdl-str}
In this subsection we will review the theory of coCartesian fibrations over the simplicial set $N(\gop)$. We also review a model category structure on the category $\sSetsMG$ in which the fibrant objects are (essentially) coCartesian fibrations. We begin by recalling the notion of a $p$-coCartesian edge:
\begin{df}
	\label{p-CC-edge}
	Let $p:X \to S$ be an inner fibration of simplicial sets. Let $f:x \to y \in (X)_1$ be an edge in $X$. We say that $f$ is $p$-coCartesian if, for all $n \ge 2$ and every (outer) commutative diagram, there exists a (dotted) lifting arrow which makes the entire diagram commutative:
	\begin{equation}
	\xymatrix{
	\Delta^{\lbrace 0, 1 \rbrace} \ar@{_{(}->}[d] \ar[rd]^f \\
	\Lambda^0[n] \ar@{_{(}->}[d] \ar[r] & X \ar[d]^p \\
	\Delta[n] \ar[r] \ar@{-->}[ru] & S
   }
	\end{equation}
	\end{df}
\begin{rem}
	Let $M$ be a (ordinary) category equipped with a functor $p:M \to I$, then an arrow $f$ in $M$, which maps isomorphically to $I$, is coCartesian in the usual sense if and only if $f$  is $N(p)$-coCartesian in the sense of the above definition, where $N(p):N(M) \to \Delta[1]$ represents the nerve of $p$.
	\end{rem}
This definition leads us to the notion of a coCartesian fibration of simplicial sets:
\begin{df}
	\label{CC-fib}
	A map of simplicial sets $p:X \to S$ is called a \emph{coCartesian} fibration if it satisfies the following conditions:
	\begin{enumerate}
		\item $p$ is an inner fibration of simplicial sets.
		\item for each edge $p:x \to y$ of $S$ and each vertex $\ud{x}$ of $X$ with $p(\ud{x}) = x$, there exists a $p$-coCartesian edge $\ud{f}:\ud{x} \to \ud{y}$ with $p(\ud{f}) = f$.
		\end{enumerate}
	\end{df}
 A coCartesin fibration roughly means that it is upto weak-equivalence determined by a \emph{functor} from $S$ to a suitably defines $\infty$-category of $\infty$-categories. This idea is explored in detail in \cite[Ch. 3]{JL}.
Next we will review the notion of \emph{relative nerve}:
\begin{df} (\cite[3.2.5.2]{JL}). Let $D$ be a category, and $f : D \to \sSets$ a functor. The nerve of $D$ relative to $f$ is the simplicial set $N_f (D)$ whose $n$-simplices are sets consisting of:
	\begin{enumerate}
		\item[(i)] a functor $d:[n] \to D$; We write $d(i, j)$ for the image of $i \le j$ in $[n]$.
		\item[(ii)] for every nonempty subposet $J \subseteq [n]$ with maximal element $j$, a map $\tau^J:\Delta^J \to f(d(j))$,
		\item[(iii)] such that for nonempty subsets $I \subseteq J \subseteq [n]$ with respective maximal elements  $i \le j$, the following diagram commutes:
		\[
		\xymatrix{
			\Delta^I\ar[r]^{\tau^I} \ar@{_{(}->}[d] & f(d(i)) \ar[d]^{f(d(i, j))} \\
			\Delta^J\ar[r]_{\tau^J}  & f(d(j))
		}
		\]
	\end{enumerate}
\end{df}

For any $f$, there is a canonical map $p_f: N_f (D) \to N(D)$ down to the ordinary nerve of $D$, induced by the unique map to the terminal object $\Delta^0 \in \sSets$ \cite[ 3.2.5.4]{JL}. When $f$ takes values in quasi-categories, this canonical map is a coCartesian fibration.
\begin{rem}
	\label{Rel-Ner-edge}
	A vertex of the simplicial set $N_f(D)$ is a pair $(c, g)$, where $c \in Ob(D)$ and $g \in f(c)_0$.
	An edge $\ud{e}:(c, g) \to (d, k)$ of the simplicial set $N_f(D)$ consists of a pair $(e, h)$, where $e:c \to d$ is an arrow in $D$ and $h:f(e)_0(g) \to k$ is an edge of $f(d)$.
	\end{rem}
An immidiate consequence of the above definition is the following proposition:
\begin{prop}
	\label{Rel-Ner-isom-func-val}
	Let $f:D \to \sSets$ be a functor, then the fiber of $p_f:N_f(D) \to N(D)$ over any $d \in Ob(D)$ is isomorphic to the simplicial set $f(d)$.
\end{prop}
 We now recall a result which will be used in the last section of this paper:
\begin{thm}
	\cite[Thm. 3.2.5.18.]{JL} The relative nerve functor $\mRN{\bullet}$ has a left adjoint $\mRNL{\bullet}$. The adjoint pair $(\mRNL{\bullet}, \mRN{\bullet})$ is a Quillen equivalence between the coCartesian model category $\ccMdl{\nGop}$ and the strict $JQ$-model category of marked $\gSs$.
	\end{thm}
The latter model category in the statement of the above theorem is constructed in section \ref{EInf-Mar-QCat}.
\begin{nota}
	To each coCartesian fibration $p:X \to \nGop$ we can associate a marked simplicial set denoted $\Nt{X}$ which is composed of the pair $(X, \E)$, where $\E$ is the set of $p$-coCartesian edges of $X$
	\end{nota}
\begin{nota}
Let $X, Y$ be two objects in $\sSetsMQ$. We will denote by $\Flmap{X}{Y}{\nGop} \subseteq \Flmap{X}{Y}{\sSetsM}$ and $\Shmap{X}{Y}{\nGop} \subseteq \Shmap{X}{Y}{\sSetsM}$ the simplicial subsets 
whose vertices are those maps which are compatible with the projections to $\nGop$.
\begin{df}
	\label{CC-Eq}
	A morphism $F:X \to Y$ in the category $\sSetsMG$ is called a \emph{coCartesian}-equivalence if for each coCartesian fibration $Z \to \nGop$, the indued simplicial map
	\[
	\Flmap{F}{\Nt{Z}}{\sSetsM}:\Flmap{Y}{\Nt{Z}}{\sSetsM} \to \Flmap{X}{\Nt{Z}}{\sSetsM}
	\]
	is a categorical equivalence of simplicial-sets(quasi-categories).
	\end{df}

\end{nota}
Next we will recall a model category structure on the overcategory $\sSetsMG$ from \cite[Prop. 3.1.3.7.]{JL} in which fibrant objects are (essentially) coCartesian fibrations.
\begin{thm}
	\label{CC-Mdl-Str}
	There is a left-proper, combinatorial model category structure on the category $\sSetsMG$ in which a morphism is 
	\begin{enumerate}
		\item a cofibration if it is a monomorphism when regarded as a map of simplicial sets.
		\item a weak-equivalences if it is a coCartesian equivalence.
		\item a fibration if it has the right lifting property with respect to all maps which are simultaneously cofibrations and weak-equivalences. 
	\end{enumerate}
\end{thm}

We have defined a function object for the category $\sSetsMG$ above.
The simplicial set $\Flmap{X}{Y}{\gop}$ has verices, all maps from $X$ to $Y$ in $\sSetsMG$. An $n$-simplex in $\Flmap{X}{Y}{\gop}$ is a map $\Fl{\Delta[n]} \times X \to Y$ in $\sSetsMG$, where $\Fl{\Delta[n]} \times (X, p) = (\Fl{\Delta[n]} \times X, pp_2)$, where $p_2$ is the projection $\Fl{\Delta[n]} \times X \to X$. The enriched category $\sSetsMG$ admits tensor and cotensor products. The \emph{tensor product} of an object $X = (X, p)$ in $\sSetsMG$ with a simplicial set $A$ is the objects
\[
\Fl{A} \times X = (\Fl{A} \times X, pp_2).
\]
The \emph{cotensor product} of $X$ by $A$ is an object of $\sSetsMG$ denoted $\expG{A}{X}$. If $q:\expG{A}{X} \to \Sh{\nGop}$ is the structure map, then a simplex $x:\Fl{\Delta[n]} \to \expG{A}{X}$ over a simplex $y = qx:\Delta[n] \to \Sh{\nGop}$ is a map $\Fl{A} \times (\Fl{\Delta[n]}, y) \to (X, p)$. The object $(\expG{A}{X}, q)$ can be constructed by the following pullback square in $\sSetsM$:
\begin{equation*}
\xymatrix{
	\expG{A}{X} \ar[r] \ar[d]_q & \mapMS{\Fl{A}}{X} \ar[d]^{\mapMS{\Fl{A}}{p}} \\
	\Sh{\nGop} \ar[r] & \mapMS{\Fl{A}}{\Sh{\nGop}}
}
\end{equation*}
where the bottom map is the diagonal. There are canonical isomorphisms:
\begin{equation}
\Flmap{\Fl{A} \times X}{Y}{D} \cong \left[A, \Flmap{X}{Y}{D} \right] \cong \Flmap{X}{\expG{A}{Y}}{D}
\end{equation}
\begin{rem}
	\label{Simp-Mdl-Cat}
	The coCartesian model category structure on $\sSetsMG$ is a simplicial model category structure with the simplicial Hom functor:
	\[
	\Shmap{-}{-}{\gop}:\sSetsMG^{op} \times \sSetsMG \to \sSets.
	\]
	This is proved in \cite[Corollary 3.1.4.4.]{JL}. The coCartesian model category structure is a $\sSetsQ$-model category structure with the function object given by:
	\[
	\Flmap{-}{-}{\gop}:\sSetsMG^{op} \times \sSetsMG \to \sSets.
	\]
	This is remark \cite[3.1.4.5.]{JL}.
\end{rem}
\begin{rem}
	\label{enrich-mar-sSets}
	The coCartesian model category is a $\sSetsMQ$-model category with the Hom functor:
	\[
	\mapGenM{-}{-}{D}:\sSetsMGen{D}^{op} \times \sSetsMGen{D} \to \sSetsM.
	\]
	This follows from \cite[Corollary 3.1.4.3]{JL} by taking $S = N(D)$ and $T = \Delta[0]$, where $S$ and $T$ are specified in the statement of the corallary.
\end{rem}

%% file: StrJQMdlStr.tex
\section[Strict JQ-model structure]{The strict JQ-model category structure on $\gSs$}
\label{str-mdl-str}
 Schwede introduced two model category structures on the category of (normalized) $\gSs$ which he
 called the \emph{strict Q-model category} structure and the
 \emph{stable Q-model category} structure in \cite{schwede}.
 The strict Q-model category structure is obtained by restricting the
 projective model category structure on the functor category $[\gop, \pSSets]$, where the codomain category $\pSSets$ is endowed with the Kan model category structure. In
 this section we study the \emph{projective} model category structure on the category of $\gSs$ namely the functor category $[\gop, \sSets]$. Following Schwede we will refer to this projective model category as the \emph{strict JQ-model category}. We will show that the strict JQ-model category is a $\sSets$- model category, where $\sSets$ is endowed with the Joyal model category structure. We go on further to show that the strict JQ model category is a symmetric monoidal closed model category.
 We begin by recalling the notion of a \emph{categorical equivalence}
 of simplicial sets which is essential for defining
 weak equivalences of the desired model category structure.
 
 \begin{df}
  A morphism of simplicial sets $f:A \to B$ is called a categorical
  equivalence if for any quasi-category $X$, the induced morphism
  on the homotopy categories of mapping spaces
  \[
   ho(\map_{\sSets}(f, X)):ho(\map_{\sSets}(B, X)) \to ho(\map_{\sSets}(A, X)),
  \]
  is an equivalence of (ordinary) categories.
 \end{df}
 \begin{rem}
 
 Categorical equivalences are weak equivalences in a cofibrantly
 generated model category structure on simplicial sets called
 the \emph{Joyal model category structure} which we will denote
 by $\sSetsQ$, see \cite[Theorem 6.12]{AJ1} for the
 definition of the Joyal model category structure.
\end{rem}
 
 \begin{df}
 We call a map of $\gSs$
 \begin{enumerate}
  \item  A \emph{strict JQ-fibration} if
 it is degreewise a \emph{pseudo-fibration} i.e. a fibration of 
 simplicial sets in the Joyal model category structure on simplicial sets.
 \item A \emph{strict JQ-equivalence}
 if it is degreewise a \emph{categorical equivalence} i.e. a weak equivalence of
 simplicial sets in the Joyal model category structure on simplicial sets,
 see \cite{JL}.
 
 \end{enumerate}
 \end{df}
 
 
 \begin{thm}
 \label{strict-cat-Q-model}
  Strict JQ-equivalences, strict JQ-fibrations
  and JQ-cofibrations provide the category of $\gSs$ with a combinatorial, left-proper model category structure on the category of $\gSs$ $\gSC$.
 \end{thm}
The model structute in the above theorem follows from \cite[Proposition A 3.3.2]{JL} and the left properness is a consequence of the left properness of the Joyal model category.
\input{EnrichStrJQ}

%% file: EnrichStrJQ.tex
 \subsection[Enrichment of the JQ-model category]{Enrichment of the JQ-model category}
 The goal of this section is to show that the JQ-model category
 is a \emph{(symmetric) monoidal model category} which is enriched over itself
 in the sense of definition \ref{enrich-model-cat}. We will prove
 this in two steps, we first establish the existence of a Quillen
 bifunctor
 \[
  - \times -:\gSC \times \sSets \to \gSC,
 \]
 where the category $\gSC$ is endowed with the JQ-model category structure
 and $\sSets$ is endowed with the Joyal model category structure $\sSetsQ$.
 Then we will use this Quillen bifunctor to prove the desired enrichment.
 We begin by reviewing the notion of a \emph{monoidal model category}. 
 
 \begin{df}
  \label{mon-model-cat}
  A \emph{monoidal model category} is a closed monoidal
  category $\C$ with a model category structure, such that
  $\C$ satisfies the following conditions:
  \begin{enumerate}
   \item The monoidal structure $\otimes:\C \times \C \to \C$
   is a Quillen bifunctor.
   
   \item Let $QS \overset{q} \to S$ be the cofibrant replacement for
   the unit object $S$, obtained by using the functorial factorization
   system to factorize $0 \to S$ into a cofibration followed by a
   trivial fibration. Then the natural map
   \[
    QS \otimes X \overset{q \otimes 1} \to S \otimes X
   \]
   is a weak equivalence for all cofibrant $X$. Similarly, the natural
   map $X \otimes QS \overset{1 \otimes q} \to X \otimes S$ is a weak equivalence
   for all cofibrant $X$.

  \end{enumerate}

 \end{df}

 \begin{ex}
 \label{sSets-Q-monoidal}
  The model category of simplicial sets with the Joyal model category
  structure, $\sSetsQ$ is a monoidal model category.
 \end{ex}
 
 \begin{ex}
 \label{stable-Q-model-monoidal}
  The stable Q-model category is a monoidal model category
  with respect to the smash product defined in \cite{lydakis}.
 \end{ex}

 \begin{df}
  \label{enrich-model-cat}
  Let $\bigS$ be a monoidal model category. An \emph{$\bigS$-enriched
  model category} is an $\bigS$ enriched category $\bigA$ equipped with
  a model category structure (on its underlying category) such that
\begin{enumerate}

\item The category $\bigA$ is tensored and cotensored over $\bigS$.
 
\item There is a Quillen adjunction of two variables, (see definition
   \ref{Q-adj-2-var}), 
\[
\left(\otimes, \bhom_{\bigA}, \map_{\bigA}, \phi, \psi \right):
  \bigA \times \bigS \to \bigA.
\]
\end{enumerate}
When $\bigA$ is itself a monoidal model category which is also an
$\bigA$-enriched model category, we will say that
$\bigA$ is \emph{enriched over itself as a model category}.
  
 \end{df}
 
 \begin{ex}
 Both strict and stable Q-model category structures, constructed in \cite{schwede}, on the category
 $\gSC$ are simplicial, i.e. both strict and stable Q- model categories are $\sSetsK$-enriched
 model categories.
 \end{ex}
 \begin{rem}
 The strict JQ-model category structure is NOT simplicial.
 \end{rem}
 
 For each pair $(F, K)$, where $F \in Ob(\gSC)$ and $K \in Ob(\sSets)$,
 one can construct a $\gS$ which we denote by $\TensP{F}{K}{}$ and which is defined
 as follows:
 \[
  (\TensP{F}{K}{})(n^+) := F(n^+) \times K,
 \]
 where the product on the right is taken in that category
 of simplicial sets. This construction is functorial in both variables.
 Thus we have a functor
 \begin{equation*}
 \label{ten-over-sSets} 
  \TensP{-}-{}{}:\gSC \times \sSets \to \gSC. 
 \end{equation*}
  Now we will define a couple of function objects for the category $\gSC$.
 The first function object enriches the category $\gSC$ over
 $\sSets$ \emph{i.e.} there is a bifunctor
 \[
 \MapC{-}{-}{\gSC}:\gSC^{op} \times \gSC \to \sSets
 \]
 which assigns to each pair of objects $(X, Y) \in Ob(\gSC) \times Ob(\gSC)$, a simplicial set
 $\MapC{X}{Y}{\gSC}$ which is defined in degree zero as follows:
 \[
\MapC{X}{Y}{\gSC}_0 := \gSC(X, Y)
 \]
 and the simplicial set is defined in degree $n$ as follows:
 \[
\MapC{X}{Y}{\gSC}_n := \gSC(\TensP{X}{\Delta[n]}{}, Y)
 \]
 For any $\gS$ $X$, the functor $\TensP{X}{-}{}:\sSets \to \gSC$ is
 left adjoint to the functor $\MapC{X}{-}{\gSC}:\gSC\to \sSets$. The \emph{counit} of this adjunction
 is the evaluation map $ev:\TensP{X}{\MapC{X}{Y}{\gSC}}{} \to Y$
 and the \emph{unit} is the obvious simplicial map $K \to \MapC{X}{\TensP{X}{K}{}}{\gSC}$.

 To each pair of objects $(K, X) \in Ob(\sSets) \times Ob(\gSC)$ we can define a $\gS$ $\bHom{K}{X}{\gSC}$, in degree $n$, as follows:
 \[
 (\bHom{K}{X}{\gSC})(n^+) := [K, X(n^+)] \ .
 \]
 This assignment
 is functorial in both variable and therefore we have a bifunctor
 \[
 \bHom{-}{-}{\gSC}:\sSets^{op} \times \gSC \to \gSC.
 \]
 For any $\gS$ $X$, the functor $\bHom{-}{X}{\gSC}:\sSets \to \gSC^{op}$ is
 left adjoint to the functor $\MapC{-}{X}{\gSC}:\gSC^{op} \to \sSets$. 
 The following proposition summarizes the above discussion.
\begin{prop}
\label{two-var-adj-cat-gcat}
There is an adjunction of two variables
\begin{equation}
\label{two-var-adj-gcat}
(\TensP{-}{-}{}, \bHom{-}{-}{\gSC}, \MapC{-}{-}{\gSC}) : \gSC \times \sSets
\\  \to \gSC.
\end{equation}

\end{prop}

\begin{thm}
  \label{enrich-GamCAT-CAT}
  The strict model category of $\gSs$, $\gSC$, is a $\sSetsQ$- model category.
 \end{thm}
 \begin{proof}
  We will show that the adjunction of two variables \eqref{two-var-adj-gcat}
  is a Quillen adjunction for the strict $JQ$-model category structure
 on $\gSC$ and the model category $\sSetsQ$.
  In order to do so, we will verify condition
 (2) of Lemma \ref{Q-bifunctor-char}. Let $g:K \to L$ be a cofibration
 in $\sSets$ and let $p:Y \to Z$ be a strict fibration of $\gSs$,
 we have to show that the induced map
 \[
  \bhom^{\Box}_{\gSC}(g, p):\bHom{L}{Y}{\gSC} \to \bHom{L}{Z}{\gSC}
  \underset{\bHom{K}{Z}{\gSC}} \times \bHom{K}{Y}{\gSC}
 \]
 is a fibration in $\gSC$ which is acyclic if either of $g$ or $p$ is
 acyclic. It would be sufficient to check that the above morphism is degreewise
 a fibration in $\sSetsQ$, i.e. for all $n^+ \in \gop$, the morphism
 \begin{equation*}
  \bhom^{\Box}_{\gSC}(g, p)(n^+) = \bhom^{\Box}_{\sSets}(g, p(n^+)):  Y(n^+)^L  \to
  Z(n^+)^L\underset{ Z(n^+)^K} \times  Y(n^+)^K,
 \end{equation*}
 is a fibration in $\sSetsQ$. This follows from the observations that the simplicial morphism $p(n^+):Y(n^+) \to Z(n^+)$
 is a fibration in $\sSetsQ$ and the model category 
 $\sSetsQ$ is a cartesian closed model category whose internal Hom is provided by the bifunctor ${-}^-:\sSets \times \sSets \to \sSets$.
 \end{proof}
 
 Let $X$ and $Y$ be two $\gSs$, the \emph{Day convolution product} of $X$ and $Y$ denoted by $X \ast Y$ is defined as follows:
 \begin{equation}
 \label{Day-Con-prod}
 X \ast Y(n^+) := \int^{(k^+, l^+) \in \gop} \gop(k^+\wedge l^+, n^+) \times X(k^+) \times Y(l^+).
 \end{equation}
 Equivalently, one may define the Day convolution product of $X$ and $Y$ as the left Kan extension of their \emph{external tensor product} $X \overline \times Y$ along the smash product functor
 \[
 - \wedge - :\gop \times \gop \to \gop.
 \]
 we recall that the external tensor product $X \overline \times Y$ is a bifunctor
 \begin{equation*}
 X \overline \times Y:\gop \times \gop \to \sSets
 \end{equation*}
 which is defined on objects by 
 \[
 X \overline \times Y(m^+, n^+) = X(m^+) \times Y(n^+).
 \]
 
 It follows from
 \cite[Thm.]{} that the functor $- \ast \gn{n}$ has a right adjoint which we denote by $-(n^+ \wedge -):\gSC \to \gSC$. We will denote the $\gS$ $-(n^+ \wedge -)(X)$ by $X(n^+ \wedge -)$ and define it by the following composite:
 \begin{equation}
 \label{defn-X-n-wedge}
 \gop \overset{n^+ \wedge -} \to \gop \overset{X} \to \sSets.
 \end{equation}
 The following proposition sums up this observation:
 \begin{prop}
 	\label{rt-adjs-DayCP}
 	There is a natural isomorphism
 	\[
 	\phi: -(n^+ \wedge -) \cong \MGCat{\gn{n}}{-}.
 	\]
 	In particular,
 	for each $\gS$ $X$ there is an isomorphism of $\gSs$
 	\[
 	\phi(X):X(n^+ \wedge -) \cong \MGCat{\gn{n}}{X}.
 	\]
 	
 \end{prop}
\begin{proof}
 Consider the functor $n^+ \wedge -:\gop \to \gop$. We observe that a
 \emph{Left Kan extension} of $\gn{1}:\gop \to \sSets$ along $n^+ \wedge -$
 is the $\gS$ $\gn{n}:\gop \to \sSets$. This implies that we have the following bijection
 \begin{equation*}
 \gSC(\gn{n}, X) \cong \gSC(\gn{1}, X(n^+ \wedge -)).
 \end{equation*}
 We observe that this natural bijection extends to a natural isomorphism of $\gSs$:
  \begin{equation*}
  \MGCat{\gn{n}}{X} \cong \MGCat{\gn{1}}{ X(n^+ \wedge -)}.
 \end{equation*}
	
\end{proof}
\begin{prop}
\label{GCAT-SM}
The category of all $\gSs$ $\gSC$ is a symmetric monoidal category under the Day convolution product \eqref{Day-Con-prod}.
The unit of the symmetric monoidal structure is the representable $\gS$ $\gn{1}$.
\end{prop}
Next we define an internal function object of the category $\gSC$
which we will denote by
\begin{equation}
\label{Int-Map-GCAT}
 \MGCat{-}{-}:\gSC^{op} \times \gSC \to \gSC.
 \end{equation}
 Let $X$ and $Y$ be two $\gSs$, we define the $\gSC$ $\MGCat{X}{Y}$ as follows:
 \begin{equation*}
 \MGCat{X}{Y}(n^+) := \MapC{X \ast \gn{n}}{Y}{\gSC}.
 \end{equation*}
 \begin{prop}
 \label{closed-SM-cat-GCat}
 The category $\gSC$ is a closed symmetric monoidal category under the Day convolution product. The internal Hom is given by the bifunctor \eqref{Int-Map-GCAT} defined above.
 \end{prop}
 The above proposition implies that for each $n \in \Nat$
 the functor $- \ast \gn{n}:\gSC \to \gSC$ has a right adjoint $\MGCat{\gn{n}}{-}:\gSC \to \gSC$.

 The next theorem shows that the strict model category $\gSC$ is compatible with the Day convolution product.
 \begin{thm}
 \label{SM-closed-mdl-str-GCat}
 The strict JQ-model category $\gSC$ is a symmetric monoidal closed model category under the Day convolution product.
 \end{thm}
 \begin{proof}
 Using the adjointness which follows from proposition \ref{closed-SM-cat-GCat} one can show that if a map $f:U \to V$ is a (acyclic) cofibration in the strict JQ-model category $\gSC$ then the induced map
 $f \ast \gn{n}:U \ast \gn{n} \to V \ast \gn{n}$ is also a (acyclic) cofibration in the strict JQ-model category for all $n \in \Nat$. By $(3)$ of Lemma \ref{Q-bifunctor-char} it is sufficient to show that whenever
 $f$ is a cofibration and $p:Y \to Z$ is a fibration then the map
 \[
    \MGBoxCat{f}{p}:\MGCat{V}{Y} \to \MGCat{V}{Z} \underset{\MGCat{U}{Z}}\times \MGCat{U}{Y}.
 \]
   is a fibration in $\gSC$ which is acyclic if either $f$ or $p$ is a
   weak equivalence. The above map is a (acyclic) fibration if and only if the simplicial map
\begin{multline*}
 \MGBoxCat{f \ast \gn{n}}{p}(n^+): \MapC{V \ast \gn{n}}{Y}{\gSC} \to \\
  \MapC{V \ast \gn{n}}{Z}{\gSC}
  \underset{\MapC{U \ast \gn{n}}{Z}{\gSC}} \times \MapC{U \ast \gn{n}}{Y}{\gSC}
\end{multline*}

    is a (acyclic) fibration in $\sSetsQ$ for all $n \in \Nat$. Since $f \ast \gn{n}$ is a cofibration (which is acyclic whenever $f$ is acyclic as observed above) therefore it follows from theorem \ref{enrich-GamCAT-CAT} that the simplicial map $\MGBoxCat{f \ast \gn{n}}{p}(n^+)$ is an (acyclic) fibration of simplicial sets for all $n \in \Nat$.
 \end{proof}

The following corollary is an easy consequence of the above theorem
and we leave the proof as an exercise for the interested reader.

\begin{coro}
Let $F'$ be a Q-cofibrant $\gS$ and $p:F \to G$ is be a
strict JQ-fibration. Then the morphism induced by $p$ on the
function objects
\[
\MGS{F'}{p}:\MGS{F'}{F} \to \MGS{F'}{G}
\]
is a strict JQ-fibration.
\end{coro} 

\begin{df}
	A morphism in $\gSC$ is called a \emph{trivial fibration} of $\gSs$ if it has the right lifting property with respect to all maps in the following class of maps
	\[
	\lbrace \gn{n} \times f ; \ f \text{is a simplicial monomorphism and } n \ge 0 \rbrace
	\]
	\end{df}

\begin{prop}
	A trivial fibration is a strict $JQ$ equivalence.
	\end{prop}
\begin{proof}
		Let $p:X \to Y$ be a trivial fibration of $\gSs$ and $f:A \to B$ be a simplicial monomorphism then whenever the outer diagram commutes in the following diagram:
	\[
	\xymatrix{
		\gn{n} \times A \ar[r] \ar[d]_{\gn{n} \times f} & X \ar[d]^p \\
		\gn{n} \times B \ar[r] \ar@{-->}[ru] & Y
	}
	\]
	there exists a dotted arrow which makes the whole diagram commutative, for each $n \ge 0$. By adjointness, we get the following commutative diagram
	in the category of simplicial sets:
	\[
	\xymatrix{
		A \ar[r] \ar[d]_{f} & \MapC{\gn{n}}{X}{\gSC} \ar[d]^{\MapC{\gn{n}}{p}{\gSC}} \\
		B \ar[r] \ar@{-->}[ru] & \MapC{\gn{n}}{Y}{\gSC}
	}
	\]
	We observe that the map
	\[
	\MapC{\gn{n}}{X}{\gSC}:\MapC{\gn{n}}{X}{\gSC} \to  \MapC{\gn{n}}{Y}{\gSC}
	\]
	is the same as the simplicial map $p(n^+):X(n^+) \to Y(n^+)$ upto isomorphism namely we have the following commutative diagram:
	\[
	\xymatrix{
		A \ar[r] \ar[d]_{f} & \MapC{\gn{n}}{X}{\gSC} \ar[d]^{\MapC{\gn{n}}{p}{\gSC}}
		\ar[r]^{\cong} & X(n^+) \ar[d]^{p(n^+)} \\
		B \ar[r] \ar@{-->}[ru] & \MapC{\gn{n}}{Y}{\gSC} \ar[r]_{\cong} & Y(n^+)
	}
	\]
	 This observation and the above simplicial commutative diagram together imply that for each $n \ge 0$, the simplicial map $p(n^+)$ has the right lifting property with respect to all simplicial monomorphisms, in other words $p(n^+)$ is a trivial fibration of simplicial sets. By \cite[Prop. 1.22]{AJ1},
	this implies that the simplicial map $p(n^+)$ being a weak equivalence in the Joyal model category of simplicial sets. Thus we have shown that $p$ is a strict $JQ$-equivalence of $\gSs$.
	\end{proof}

 \begin{prop}
	A strict $JQ$-fibration is a trivial fibration of $\gSs$ if and only if it is a strict $JQ$ equivalence.
\end{prop}

%% file: EinMdlStrQCat.tex
\section[The model category of coherently commutative monoidal quasicategories]{The JQ-model category}

\label{EInf-QCat}
  The objective of this section is to construct a new model
  category structure on the category $\gSC$. This new model
  category is obtained by localizing the strict JQ-model category
 defined above. We will refer to this new model category structure
 as the \emph{model category structure of coherently commutative monoidal quasi-categories}. The guiding principle of this new model structure is to endow its homotopy category with a semi-additive structure. In other words we want this new
 model category structure to have finite \emph{homotopy biproducts}.  We go on further to show that this new model category is symmetric monoidal closed with respect to
 the \emph{Day convolution product}, see \cite{Day2}.    We begin by recalling the notion of
 a \emph{left Bousfield localization}:
 
 \begin{df}
  Let $\M$ be a model category and let $\S$ be a class of maps in $\M$.
  The left Bousfield localization of $\M$ with respect to $\S$
  is a model category structure $L_\S\M$ on the underlying category of $\M$
  such that
  \begin{enumerate}
   \item The class of cofibrations of $L_\S\M$ is the same as the
   class of cofibrations of $\M$.
   
   \item A map $f:A \to B$ is a weak equivalence in $L_\S\M$ if it is an $\S$-local equivalence,
   namely, for every fibrant $\S$-local object $X$, the induced map on homotopy
   function complexes
   \[
    f^\ast:Map_{\M}^h(B, X) \to Map_{\M}^h(A, X)
   \]
   is a homotopy equivalence of simplicial sets. Recall
   that an object $X$ is called fibrant $\S$-local if $X$ is fibrant
    in $\M$ and for every element
   $g:K \to L$ of the set $\S$, the induced map on
   homotopy function complexes
   \[
    g^\ast:Map_{\M}^h(L, X) \to Map_{\M}^h(K, X)
   \]
   is a weak homotopy equivalence of simplicial sets.
    
  \end{enumerate}
where $\HMapC{-}{-}{\M}$ is the simplicial function object associated with the strict model category $\M$, see \cite{DK80}, \cite{DK1980} and \cite{DK3}.
 \end{df}
 
 We want to construct a left Bousfield localization of
 the strict model category of $\gSs$. For each pair $k^+, l^+ \in \gop$,
 we have the obvious \emph{projection maps} in $\gSC$
 \[
  \delta^{k+l}_k:(k+l)^+ \to k^+ \ \ \ \ and \ \ \ \ \delta^{k+l}_l:(k+l)^+ \to l^+.
 \]
 The maps
 \[
 \gop(\delta^{k+l}_k,-):\Gamma^{k} \to \Gamma^{k+l} \ \ \ \ and \ \ \ \ 
 \gop(\delta^{k+l}_l,-):\Gamma^{l} \to \Gamma^{k+l} 
 \]
 induce a map of $\gSs$ on the coproduct which we denote as follows:
 \[
  h_k^l:\Gamma^l \sqcup \Gamma^l \to \Gamma^{l+k}.
 \]
 
 We now define a class of
 maps $\E_\infty\S$ in $\gSC$:
 \begin{equation*}
  \E_\infty\S := \lbrace h_k^l:\Gamma^l \sqcup \Gamma^l \to \Gamma^{l+k}:
  l, k \in \mathbb{Z}^+ \rbrace
 \end{equation*}
 \begin{df}
  We call a $\gS$ $X$ a $(\Delta \times \E_\infty\S)$-\emph{local object}
  if it is a fibrant object in the strict $JQ$-model category and for each map $h_k^l \in \E_\infty\S$, the induced simplicial map
  \begin{multline*}
  \HMapC{\Delta[n] \times h_k^l}{X}{\gSC}: \HMapC{\Delta[n] \times \Gamma^{k+l}}{X}{\gSC} \to \\
  \HMapC{\Delta[n] \times (\Gamma^l \sqcup \Gamma^l)}{X}{\gSC},
  \end{multline*}
 is a homotopy equivalence of simplicial sets for all $n \ge 0$ where $\HMapC{-}{-}{\gSC}$ is the simplicial function complexes associated with the strict model category $\gSC$, see \cite{DK80}, \cite{DK1980} and \cite{DK3}.
 \end{df}
  Appendix \ref{Cat-Local} tell us that a model for $\HMapC{X}{Y}{\gSC}$ is the Kan complex
 $J(\MapC{X}{Y}{\gSC})$ which is the maximal kan complex contained in the quasicategory $\MapC{X}{Y}{\gSC}$.

  The following proposition gives a characterization of
 $\E_\infty\S$-local objects
  \begin{prop}
  \label{char-CCMC}
 \begin{sloppypar}
 A $\gS$ $X$ is a $(\Delta \times \E_\infty\S)$-local object in $\gSC$ if and only if it satisfies the Segal condition namely the functor
  \end{sloppypar}
 \begin{equation*}
 (X(\partition{(k+l)}{k}), X(\partition{(k+l)}{l})):X((k+l)^+) \to X(k^+) \times X(l^+)
 \end{equation*}
 is an equivalence of categories for all $k^+, l^+ \in \Ob(\gop)$.
 \end{prop}
 \begin{proof}
 We begin the proof by observing that each element of the set $\E_\infty\S$ is a map of $\gSs$ between cofibrant $\gSs$. Theorem \ref{char-lo-QCat-en} implies that $X$ is a $(\Delta \times \E_\infty\S)$-local object if and only if the following simplicial map
 \begin{equation*}
 \MapC{h^k_l}{X}{\gSC}:\MapC{\gn{(k+l)}}{X}{\gSC} \to \MapC{\gn{k} \sqcup \gn{l}}{X}{\gSC}
 \end{equation*}
 is a categorical equivalence of quasi-categories.
 
 We observe that we have the following commutative square in $\sSetsQ$
 \begin{equation*}
 \xymatrix@C=24mm{
\MapC{\gn{(k+l)}}{X}{\gSC}  \ar[d]_{\cong}\ar[r]^{ \MapC{h^k_l}{X}{\gSC}} & \MapC{\gn{k} \sqcup \gn{l}}{X}{\gSC}   \ar[d]^{\cong} \\
 X((k+l)^+) \ar[r]_{(X(\partition{(k+l)}{k}), X(\partition{(k+l)}{l}))} & X(k^+) \times X(l^+)
 }
 \end{equation*}
 This implies that the functor $(X(\partition{(k+l)}{k}), X(\partition{(k+l)}{l}))$ is an equivalence of categories if and only if the functor $ \MapC{h^k_l}{X}{\gSC}$ is an equivalence of categories.
  \end{proof}

\begin{df}
 \label{CCMC}
 We will refer to a $(\Delta \times \E_\infty\S)$-local object as a \emph{coherently commutative monoidal quasi-category}.
 \end{df}
\begin{df}
	Let $X$ be a coherently commutative monoidal quasi-category. We will refer to the homotopy category of the quasi-category $X(1^+)$, $ho(X(1^+))$, as the \emph{homotopy category} of $X$ and denote it by $ho(X)$.
\end{df}

\begin{prop}
	\label{homotopy-cat-perm}
	The homotopy category of a coherently commutative monoidal quasi-category is a permutative category.
\end{prop}
 
 \begin{df}
 A morphism of $\gSs$ $F:X \to Y$ is a $(\Delta \times \E_\infty\S)$-local equivalence if for each coherently commutative monoidal category $Z$
 the following simplicial map
 \[
 \HMapC{F}{Z}{\gSC}:\HMapC{Y}{Z}{\gSC} \to \HMapC{X}{Z}{\gSC}
 \]
 is a homotopy equivalence of simplicial sets.
 \end{df}
 \begin{prop}
  \label{char-CCME}
 \begin{sloppypar}
 A morphism between two cofibrant $\gSs$ $F:X \to Y$ is an $(\Delta \times \E_\infty\S)$-local equivalence if and only if the simplicial map  \end{sloppypar}
 \begin{equation*}
 \MapC{F}{Z}{\gSC}:\MapC{Y}{Z}{\gSC} \to \MapC{X}{Z}{\gSC}
  \end{equation*}
 is an equivalence of quasi-categories for each coherently commutative monoidal quasi-category $Z$.
 \end{prop}
\begin{proof}
	Let us first assume that $F:X \to Y$ is an $(\Delta \times \E_\infty\S)$-local equivalence. Then for any coherently commutative monoidal quasi-category $Z$ the following simplicial map
	\[
	\HMapC{F}{Z}{\gSC}: \HMapC{Y}{Z}{\gSC} \to \HMapC{X}{Z}{\gSC}
	\]
	is a homotopy equivalence of Kan complexes. We observe that for each $n >0$, the $\gS$ $Z^{\Delta[n]}$ is a coherently commutative monoidal quasi-category because it satisfies the Segal condition, see \ref{char-CCMC}, namely
	we have the following diagram in which the first map is an equivalence of quasi-categories
	\[
	Z((k+l)^+)^{\Delta[n]} \to (Z(k^+) \times Z(l^+))^{\Delta[n]} \cong Z(k^+)^{\Delta[n]} \times Z(l^+)^{\Delta[n]}.
	\]
This implies that for each $n > 0$ the following simplicial map is an equivalence of quasi-categories:
\[
 \HMapC{F}{Z^{\Delta[n]}}{\gSC}: \HMapC{Y}{Z^{\Delta[n]}}{\gSC} \to \HMapC{X}{Z^{\Delta[n]}}{\gSC}.
\]
By Proposition \ref{func-sSet-char}, we have
\[
  \HMapC{F}{Z^{\Delta[n]}}{\gSC} = J(\MapC{F}{Z^{\Delta[n]}}{\gSC}).
\]
By adjointness we have the following isomorphisms in the category of arrows  of simplicial sets:
\[
 \MapC{F}{Z^{\Delta[n]}}{\gSC} \cong \MapC{F \times \Delta[n]}{Z}{\gSC} \cong \MapC{F}{Z}{\gSC}^{\Delta[n]}
\]
Since the map $J(\MapC{F}{Z^{\Delta[n]}}{\gSC})$ is a homotopy equivalence of Kan complexes, the above isomorphisms imply that so is the simplicial map $J(\MapC{F}{Z}{\gSC}^{\Delta[n]})$. Now Lemma \ref{char-lo-QCat-en}
says that the simplicial map $\MapC{F}{Z}{\gSC}$ is an equivalence of quasi-categories.

Conversely, let us assume that the simplicial map $\MapC{F}{Z}{\gSC}$ is an equivalence of quasi-categories. Since the functor $J$ takes equivalences of quasi-categories to homotopy equivalences of Kan complexes, therefore $J(\MapC{F}{Z}{\gSC}) = \HMapC{F}{Z}{\gSC}$ is a homotopy equivalence of Kan complexes. Thus we have shown that $F:X \to Y$ is a $(\Delta \times \E_\infty \S)$-local object.
	
	\end{proof}

 \begin{df}
  We will refer to a $(\Delta \times \E_\infty \S)$-local equivalence either as an \emph{equivalence of coheretly commutative monoidal categories} or as a $JQ$-equivalence.
 \end{df}
 The main result of this section is about constructing a
 new model category structure on the category $\gSC$,
 by localizing the strict model category of $\gSs$ with respect to
 morphisms in the set $\E_\infty\S$. We recall the following theorem
 which will be the main tool in the construction of the
 desired model category. This theorem first appeared in an unpublished work \cite{smith}
 but a proof was later provided by Barwick in \cite{CB1}.
 \begin{thm} \cite[Theorem 2.11]{CB1}
 \label{local-tool}
 If $\M$ is a combinatorial model category and $\S$ is a small
set of homotopy classes of morphisms of $\M$, the left Bousfield localization $L_\S\M$ of
$\M$ along any set representing $\S$ exists and satisfies the following conditions.
\begin{enumerate}
\item The model category $L_\S\M$ is left proper and combinatorial.
\item As a category, $L_\S\M$ is simply $\M$.
\item The cofibrations of $L_\S\M$ are exactly those of $\M$.
\item The fibrant objects of $L_\S\M$ are the fibrant $\S$-local objects $Z$ of $\M$.
\item The weak equivalences of $L_\S\M$ are the $\S$-local equivalences.
\end{enumerate}
\end{thm}
\begin{thm}
 \label{loc-semi-add}
 There is a closed, left proper, combinatorial model category structure on
 the category of $\gSs$, $\gSC$, in which
 \begin{enumerate}
 \item The class of cofibrations is the same as the class of
 JQ-cofibrations of $\gSs$.
 \item The weak equivalences are equivalences of coherently commutative monoidal quasi-categories..
 \end{enumerate}
 \begin{sloppypar}
 An object is fibrant in this model category if and only if it is a
  coherently commutative monoidal category. A fibration between two coherently commutative monoidal quasi-categories is a strict $JQ$-equivalence.
   \end{sloppypar}
 \end{thm}
 \begin{proof}
 The strict model category of $\gSs$ is a combinatorial
 model category therefore the existence of the model structure
 follows from theorem \ref{local-tool} stated above.
The last statement follows from $(1)$.
 \end{proof}
 \begin{nota}
 The model category constructed in theorem \ref{loc-semi-add} will
 be called the model category of $\EinQCs$.
 \end{nota}
  \begin{sloppypar}
 The rest of this section is devoted to proving that the model
 category of $\EinQCs$  is a symmetric monoidal closed model category.
 In order to do so we will need some general results which we
 state and prove now.
 \end{sloppypar}
 
\begin{prop}
 \label{criterion-acy-cof}
 A cofibration, $f:A \to B$, between cofibrant objects in a model category $\C$ is
 a weak equivalence in $\C$ if and only if it has the right
 lifting property with respect to all fibrations between fibrant
 objects in $\C$.
 \end{prop}
 \begin{proof}
 The unique terminal map $B \to \ast$ can be factored into
 an acyclic cofibration $\eta_B:B \to R(B)$ followed by a fibration
 $R(B) \to \ast$. The composite map $\eta_B \circ f$ can again be
 factored as an acyclic cofibration followed by a fibration $R(f)$ as shown
 in the following diagram:
 \begin{equation*}
  \label{fact-comp-cof}
 \xymatrix{
 A \ar[d]_f \ar[r]^{\eta_A} & R(A) \ar[d]^{R(f)} \\
 B \ar@{.>}[ru] \ar[r]_{\eta_B} &R(B)
 }
 \end{equation*}
 Since $B$ is fibrant and $R(f)$ is a fibration, therefore $R(A)$
 is a fibrant object in $\C$. Thus $R(f)$ is a fibration between
 fibrant objects in $\C$ and now by assumption, the 
 dotted arrow exists which makes the whole diagram commutative.
 Since both $\eta_A$ and $\eta_B$ are acyclic cofibrations, therefore
 the two out of six property of model categories implies that
 the map $F$ is a weak-equivalence in the model category $\C$.
 
 \end{proof}

 \begin{prop}
 \label{expn-ho-prod}
 \begin{sloppypar}
 Let $X$ be a $\EinQC$, then for each $n \in \Nat$,
 the $\gS$ $X(n^+ \wedge -)$ is also a $\EinQC$.
 \end{sloppypar}
\end{prop}
\begin{proof}
 We begin by observing that
 $X(n^+ \wedge -)(1^+) = X(n^+)$ and since $X$ is fibrant,
 the pointed category $X(n^+)$ is equivalent to $\overset{n}
 {\underset{1} \prod} X(1^+)$. Notice that the isomorphisms $(n^+ \wedge (k+l)^+) \cong 
 \overset{n} {\underset{1} \vee} (k+l)^+ \cong (\overset{n} {\underset{1}
 \vee} k^+) \vee (\overset{n} {\underset{1} \vee} l^+) \cong ((\overset{n}
 {\underset{1}  \vee} k^+) + (\overset{n} {\underset{1} \vee} l^+))$.
 The two projection maps $\delta^{k+l}_k:(k+l)^+ \to k^+$ and $\delta^{k+l}_l:(k+l)^+ \to l^+$
 induce an equivalence of categories $X((\overset{n} {\underset{1}
 \vee} k^+) + (\overset{n} {\underset{1} \vee} l^+)) \to
 X(\overset{n} {\underset{1} \vee} k^+) \times X(\overset{n} {\underset{1}
 \vee} l^+)$. Composing with the isomorphisms above, we get
 the following equivalence of pointed simplicial sets
 $X(n^+ \wedge -)((k + l)^+) \to X(n^+ \wedge -)(k^+) \times
 X(n^+ \wedge -)(l^+)$.
 \end{proof}
 
 \begin{coro}
  For each coherently commutative monoidal category $X$,
  the mapping object $\MGCat{\gn{n}}{X}$ is also a coherently commutative monoidal category for each $n \in \Nat$.
 \end{coro}
 \begin{proof}
  The corollary follows from the above proposition and proposition \ref{rt-adjs-DayCP}.
 \end{proof}

 The category $\gop$ is a symmetric monoidal category with respect
 to the smash product of pointed sets. In other words the smash product of
 pointed sets defines a
  bi-functor $- \wedge -: \Gamma^{op} \times \Gamma^{op} \to
 \Gamma^{op}$. For each pair $k^+, l^+ \in Ob(\gop)$, there are two
 natural transformations
 
 \[
 \delta^{k+l}_k \wedge -: (k+l)^+ \wedge - \Rightarrow k^+ \wedge - \ \ \ \ \text{and} \ \ \ \
 \delta^{k+l}_l \wedge -: (k+l)^+ \wedge - \Rightarrow l^+ \wedge -.
 \]
  Horizontal composition of either of these two natural transformations
  with a $\gS$ $X$
 determines a morphism of $\gSs$
 \[
  id_X \circ (\delta^{k+l}_k \wedge -) =:X(\delta^{k+l}_k \wedge -):X((k+l)^+ \wedge -) \to X(k^+ \wedge -).
 \]

 \begin{prop}
 \label{Hom-prod-fib-gCat}
 Let $X$ be an $\EinQC$, then for each pair $(k,l) \in \Nat \times \Nat$,
 the following morphism
 \[
 (X(\delta^{k+l}_k \wedge -),X(\delta^{k+l}_l \wedge -)):X((k+l)^+ \wedge -) \to
 X(k^+ \wedge -) \times X(l^+ \wedge -)
 \]
 is a strict equivalence of $\gSs$.
 \end{prop}
 
 Using the previous two propositions, we now show
 that the mapping space functor $\MGCat{-}{-}$ provides
 the homotopically correct function object when the
 domain is cofibrant and codomain is fibrant.
 
 \begin{lem}
  \label{Ein-map-obj-Cat}
  \begin{sloppypar}
 Let $W$ be a Q-cofibrant $\gS$ and let $X$ be a coherently commutative monoidal quasi-category. Then the mapping object $\MGCat{W}{X}$ is also a $\EinQC$.
 \end{sloppypar}
 \end{lem}
 \begin{proof}
 We begin by recalling that
 \[
 \MGCat{W}{X}((k+l)^+) = \MapC{W \ast \gn{k+l}}{X}{\gSC}.
 \]
 \begin{sloppypar}
 We recall that the $\gn{k+l}$ is a cofibrant $\gS$ and by assumption $W$ is also a cofibrant $\gS$ therefore it follows from
 Theorem \ref{SM-closed-mdl-str-GCat} that $W \ast \gn{k+l}$ is also a cofibrant $\gS$.
 Since $X$ is a $\EinQC$ \emph{i.e.} a fibrant object in the model category of coherently commutative monoidal quasi-categories, therefore
 it follows from Theorem \ref{SM-closed-mdl-str-GCat} that the mapping object $\MapC{W \ast \gn{k+l}}{X}{\gSC}$ is a quasi-category, for all $k, l \ge 0$.
 
 We recall that the map $h^k_l:\gn{k} \sqcup \gn{l} \to \gn{k+l}$
 is a weak equivalence in the model category of coherently commutative monoidal quasi-categories, therefore the top arrow in the following commutative diagram is a categorical equivalence of quasi-categories:
 \begin{equation*}
 \xymatrix{
 \MapC{W \ast \gn{k+l}}{X}{\gSC} \ar[rr]^{\MapC{W \ast h^k_l}{X}{\gSC} \ \ \ \ } && \MapC{(W \ast \gn{k}) \sqcup (W \ast \gn{l})}{X}{\gSC} \ar[d]^{\cong} \\
 && \MapC{W \ast \gn{k}}{X}{\gSC} \times \MapC{W \ast \gn{l}}{X}{\gSC}
}
 \end{equation*}

%
\end{sloppypar}
 \end{proof}

Finally we get to the main result of this section. All the lemmas proved above will be useful in proving the following theorem:
\begin{thm}
\label{SM-closed-CCMC}
The model category of coherently commutative monoidal quasi-categories is a symmetric monoidal closed model category under the Day convolution product.
\end{thm}
\begin{proof}
Let $i:U \to V$ be a $JQ$-cofibration and $j:Y \to Z$ be another $JQ$-cofibration. We will prove the theorem by showing that the following \emph{pushout product} morphism
\begin{equation*}
i \Box j:U \ast Z \underset{U \ast Y} \coprod V \ast Y \to V \ast Z 
\end{equation*}

is a $JQ$-cofibration which is also a $JQ$-equivalence whenever either $i$ or $j$ is a $JQ$-equivalence.
We first deal with the case of $i$ being a generating $JQ$-cofibration. The closed symmetric monoidal model structure on the strict $JQ$-model category, see theorem \ref{SM-closed-mdl-str-GCat}, implies that $i \Box j$ is a $JQ$-cofibration. Let us assume that $j$ is an acyclic $JQ$-cofibration \emph{i.e.} the $JQ$-cofibration $j$ is also a $JQ$-equivalence of coherently commutative monoidal categories. According to proposition \ref{criterion-acy-cof} the $JQ$-cofibration $i \Box j$ is a $JQ$-equivalence if and only if it has the left lifting property with respect to all strict $JQ$-fibrations of $\gSs$ between coherently commutative monoidal quasi-categories. Let $p:W \to X$ be a strict $JQ$-fibration between two coherently commutative monoidal quasi-categories. By adjointness, a (dotted) lifting arrow would exists in the following diagram
\begin{equation*}
\xymatrix{
U \ast Z \underset{U \ast Y} \coprod V \ast Y \ar[r] \ar[d] & W \ar[d]^p \\
V \ast Z \ar@{..>}[ru] \ar[r] & Y
}
\end{equation*}
if and only if a (dotted) lifting arrow exists in the following adjoint commutative diagram
\begin{equation*}
\xymatrix{
Y \ar[r] \ar[d]_j & \MGCat{V}{W} \ar[d]^{(j^*, p^*)} \\
Z \ar@{..>}[ru] \ar[r] & \MGCat{U}{X} \underset{\MGCat{U}{Y}} \times \MGCat{V}{Y}
}
\end{equation*}
The map $(j^*, p^*)$ is a strict $JQ$-fibration of $\gSs$ by lemma \ref{Q-bifunctor-char} and theorem \ref{SM-closed-mdl-str-GCat}. Further the observation that both $V$ and $U$ are $JQ$-cofibrant and the above lemma \ref{Ein-map-obj-Cat} together imply that $(j^*, p^*)$ is a strict $JQ$-fibration between coherently commutative monoidal categories and therefore a fibration in the $JQ$-model category. Since $j$ is an acyclic cofibration in the $JQ$-model category by assumption therefore the (dotted) lifting arrow exists in the above diagram. Thus we have shown that if $i$ is a $JQ$-cofibration and $j$ is a $JQ$-cofibration which is also a weak equivalence in the $JQ$-model category then $i \Box j$ is an acyclic cofibration in the $JQ$-model category.
Now we deal with the general case of $i$ being an arbitrary $JQ$-cofibration. Consider the following set:
\begin{equation*}
S = \lbrace i:U \to V | \ i \Box j \textit{ \ is an acyclic cofibration in }\gSC \rbrace
\end{equation*}
where $\gSC$ is endowed with the $JQ$-model structure. We have proved above that the set $S$ contains all generating $JQ$-cofibrations.
We observe that the set $S$ is closed under pushouts, transfinite compositions and retracts. Thus $S$ contains all $JQ$-cofibrations.
Thus we have proved that $i \Box j$ is a cofibration which is acyclic if $j$ is acyclic. The same argument as above when applied to the second argument of the Box product (\emph{i.e.} in the variable j)
shows that $i \Box j$ is an acyclic cofibration whenever $i$ is an acyclic cofibration in the $JQ$-model category.

\end{proof}

Finally we will show that the construction of the model category of coherently commutative monoidal quasi-categories achieves the goal of inducing a semi-additive structure on its homotopy category:

\begin{thm}
	\label{Ho-Cat-SA}
	The homotopy category of the model category of coherently commutative monoidal quasi-categories is semi-additive.
	\end{thm}
\begin{proof}
	In light of \cite[Prop. 6(ii)]{GS}, it is sufficient to show that for each cofibrant coherently commutative monoidal quasi-category $X$, its coproduct with itself is homotopy equivalent to its product with itself. This follows from observing that the following map is a homotopy equivalence:
	\[
	X \sqcup X \cong X \ast (\gn{1} \sqcup \gn{1}) \overset{X \ast h^1_1} \to X \ast (\gn{2}) \cong X \times X.
	\]
	\end{proof}

%% file: QuillenEq.tex
\section[Equivalence with normalized $\gSs$]{Equivalence with normalized $\gSs$}
In this section we will establish a Quillen equivalence between the model category of coherently commutative monoidal quasi-categories and the model category of strictly unital coherently commutative monoidal quasi-categories which is constructed in appendix \ref{JQ-su-ccm-qcat}.
The category of normalized $\gSs$ is equipped with a forgetful functor
\begin{equation}
\label{Forget-fun-GS}
U:\pGSC \to \gSC.
\end{equation}
This functor maps a normalized $\gS$ $X$ to the following composite
\[
\gop \overset{X} \to \pSSets \overset{U_{\sSets}} \to \sSets,
\]
where the second functor is the obvious forgetful functor which forgets the basepoint of a simplicial set. The forgetful functor $U$ has some very desirable homotopical properties: We will show in this section that $U$ preserves weak-equivalences namely it maps $JQ$-equivalences of normalized $\gSs$ to $JQ$-equivalences. This functor also preserves cofibrations even though it is a right Quillen functor.
\begin{prop}
	\label{pres-acy-fib}
	The forgetful functor $U:\pGSC \to \gSC$ preserves acyclic fibrations.
	\end{prop}
\begin{proof}
	A morphism of normalized $\gSs$ $p:X \to Y$ is an acyclic fibration in the $JQ$-model category of normalized $\gSs$ if and only if there is a lifting arrow in the following commutative diagram for each $n \in \Nat$
	\[
	\xymatrix{
   \gn{n} \wedge \del \Delta[n]^+ \ar[d] \ar[r] & X \ar[d] \\
   \gn{n} \wedge \Delta[n]^+ \ar[r] \ar@{-->}[ru] & Y
    }
 \]
 because the collection $I_\bullet = \lbrace \gn{n} \wedge \del \Delta[n]^+ \to \gn{n} \wedge \Delta[n]^+ : n \in \Nat \rbrace$ is a set of generating cofibrations for the combinatorial $JQ$-model category of normalized $\gSs$ $\pGSC$.
  By adjointness the lifting arrow exists in the above diagram if and only if a lifting arrow exists in the following (adjunct) commutative diagram in $\pSSets$
  	\[
  \xymatrix{
  	 \del \Delta[n]^+ \ar[d] \ar[r] & \MapC{\gn{n}}{X}{\pGSC} \cong X(n^+) \ar[d]^p \\
   \Delta[n]^+ \ar[r] \ar@{-->}[ru] & \MapC{\gn{n}}{Y}{\pGSC} \cong Y(n^+)
  }
  \]
  We recall the adjunction $(-)^+:\sSets \rightleftharpoons \pSSets:U_\sSets $ and observe that $U(X)(n^+) = U_\sSets(X(n^+))$. This implies that the lifting arrow in the above commutative diagram of pointed simplicial sets will exists if and only if a lifting arrow exists in the following (adjunct) commutative diagram in $\sSets$
  	\[
  \xymatrix{
  	\del \Delta[n] \ar[d] \ar[r] &  U_\sSets(X(n^+)) \ar[d]^{U_\sSets(p_n)} \\
  	\Delta[n] \ar[r] \ar@{-->}[ru] &  U_\sSets(Y(n^+))
  }
  \]
  We observe that for any normalized $\gS$ $Z$, $U_\sSets(Z(n^+)) \cong U(Z)(n^+)$. Therefore a lifting arrow exists in the above diagram if and only if a lifting arrow exists in the following commutative diagram:
  	\[
  \xymatrix{
  	\del \Delta[n] \ar[d] \ar[r] &  \MapC{\gn{n}}{U(X)}{\gSC} \ar[d]^{\MapC{\gn{n}}{U(p_n)}{\gSC}} \\
  	\Delta[n] \ar[r] \ar@{-->}[ru] &  \MapC{\gn{n}}{U(Y)}{\gSC}
  }
  \]
  By adjointness, this lifting arrow would exist if and only if there exists a lifting arrow in the following (adjunct) commutative diagram:
  \[
  \xymatrix{
  	\gn{n} \times \del \Delta[n] \ar[d] \ar[r] & U(X) \ar[d] \\
  	\gn{n} \times \Delta[n] \ar[r] \ar@{-->}[ru] & U(Y)
  }
  \]
  The collection $I_\bullet = \lbrace \gn{n} \wedge \del \Delta[n]^+ \to \gn{n} \wedge \Delta[n]^+ : n \in \Nat \rbrace$ is a set of generating cofibrations for the combinatorial model category $\gSC$. Thus we have shown that the map of $\gSs$ $U(p)$ has the right lifting property with respect to the set of generating cofibrations of the $JQ$-model category and hence $U(p)$ is an acyclic fibration.

	\end{proof}

A similar argument as in the proof of the above proposition when applied to the collection of generating acyclic cofibrations of the strict $JQ$-model category of normalized $\gSs$  $\pGSC$ gives a proof of the following proposition:
\begin{prop}
	\label{pres-str-fib}
	The forgetful functor $U:\pGSC \to \gSC$ preserves strict $JQ$-fibrations.
	\end{prop}

We would like to construct a left adjoint of the functor $U$.
For a given $\gS$ $X$ we will construct another $\gS$ $X[0]$ which is  equipped with a map $\iota:X[0] \to X$.

\begin{df}
	Let $X$ be a $\gS$, the \emph{unital part} of $X$ is the constant $\gS$ $X[0]$ which is defined by
	\[
	X[0](n^+) := X(0^+)
	\]
	for all $n^+ \in Ob(\gop)$. The map $\iota$ is defined, in degree $n$ by the following simplicial map:
	\[
	\iota(n^+) := X(0_n):X(0^+) \to X(n^+),
	\]
	where $0_n:0^+ \to n^+$ is the unique map in $\gop$ between $0^+$ and $n^+$. 
\end{df}
We notice that if $X$ is a normalized $\gS$ then the unital part of $U(X)[0]$ is the terminal $\gS$. We want to use the above construction to associate with a $\gS$ a normalized $\gS$ which is equipped with a map from the original $\gS$.
\begin{df}
	Let $X$ be a $\gS$, we define another $\gS$ $U(\nor{X})$  by the following pushout square:
	\begin{equation}
	\label{normalization-GSs}
	\xymatrix{
		X[0] \ar[d] \ar[r]^\iota & X \ar[d]^{\eta_X} \\
		1 \ar[r] & U(\nor{X})
	}
	\end{equation}
	where $1$ is the terminal $\gS$. The bottom horizontal arrow in the above pushout square is an object of the category $1/\gSC$.  Since $U(\nor{X})(0^+) = \ast$
	therefore the image is an object of the category $(1/\gSC)_\bullet$, see \eqref{nor-pointed-GS}. Its image under the isomorphism of categories from the remark following \eqref{nor-pointed-GS} determines a normalized $\gS$ which we denote by $\nor{X}$ and call it the \emph{normalization of} $X$.
\end{df}
The above construction is functorial in $X$ and hence we have defined a functor $\nor{(-)}:\gSC \to \pGSC$. 
\begin{prop}
	\label{deg-wise-mono}
	For any $\gS$ $X$ the map $\iota:X[0] \to X$ defined above is a degreewise mononorphism of simplicial sets.
\end{prop}
\begin{proof}
	We want to show that for each $n^+ \in Ob(\gop)$ the simplicial map $\iota(n^+):X(0^+) \to X(n^+)$ is a monomorphism. We observe that the object $0^+$ is the zero object in $\gop$ therefore the unique composite arrow
	\[
	0^+ \to n^+ \to 0^+
	\]
	is the identity map of $0^+$, for all $n^+ \in Ob(\gop)$. This implies that the simplicial map $X(0^+) \to X(n^+) \to X(0^+)$ is the identity map of $X(0^+)$, in other words the simplicial map $\iota(n^+)$ has a left inverse which implies that the map $\iota(n^+):X(0^+) \to X(n^+)$ is a monomorphism.
\end{proof}
\begin{prop}
	\label{unit-eq-fib-obj}
	For each coherently commutative monoidal quasi-category $X$ the map of $\gSs$ $\eta_X:X \to U(\nor{X})$ defined in \eqref{normalization-GSs} is a (strict) $JQ$-equivalence.
	\end{prop}
\begin{proof}
	Since $X$ is a coherently commutative monoidal quasi-category by assumption therefore the unique terminal map $X[0] \to 1$ is a strict $JQ$-equivalence because $X(0^+)$ is homotopy equivalent to the terminal simplicial set in the Joyal model category. The $\gS$ $U(\nor{X})$ is defined as a pushout, see \eqref{normalization-GSs}, and pushouts in the category $\gS$ are degreewise therefore we have the following pushout diagram in the category $\sSets$:
	\begin{equation}
	\label{deg-wise-normalization}
	\xymatrix{
		X[0](n^+) \ar[d] \ar[r]^{\iota(n^+)} & X(n^+) \ar[d]^{\eta_X(n^+)} \\
		1 \ar[r] & U(\nor{X})(n^+)
	}
	\end{equation}
	for each $n^+ \in Ob(\gop)$. By proposition \ref{deg-wise-mono} the simplicial map $\iota(n^+)$ is a monomorphism. Since monomorphisms are cofibrations in the Joyal model category which is a left proper model category therefore a pushout of a weak equivalence along a monomorphism is a weak equivalence in Joyal model category. Thus we have shown that the map $\iota(n^+):X(n^+) \to U(\nor{X})(n^+)$ is a weak equivalence in the Joyal model category which proves that the unit map $\eta_X:X \to U(\nor{X})$ is a strict $JQ$-equivalence whenever $X$ is a coherently commutative monoidal quasi-category.
	
	\end{proof}
\begin{coro}
	\label{pres-fib-objs}
	The functor $\nor{(-)}$ takes coherently commutative monoidal quasi-categories to strictly unital coherently commutative monoidal quasi-categories.
	\end{coro}
\begin{coro}
	\label{U-pres-weak-eq}
	The functor $U:\pGSC \to \gSC$ preserves weak equivalences.
	\end{coro}
\begin{proof}
	A map $f:X \to Y$ is a $JQ$-equivalence of normalized $\gSs$ if and only if each cofibrant replacement if also the same so we may assume that $f$ is a map between cofibrant objects. It would be sufficient to show that for each coherently commutative monoidal quasi-category $Z$ the following simplicial map is an equivalence of quasi-categories:
	\[
	\MapC{U(f)}{Z}{\gSC}:\MapC{U(Y)}{Z}{\gSC} \to \MapC{U(X)}{Z}{\gSC}
	\]
	We have the following commutative diagram of simplicial mapping objects:
	\[
	\xymatrix@C=20mm{
	\MapC{U(Y)}{Z}{\gSC} \ar[r]^{^{\MapC{U(f)}{Z}{\gSC}}} \ar[d]_{\MapC{U(Y)}{\eta_Z}{\gSC}} & \MapC{U(X)}{Z}{\gSC} \ar[d]^{\MapC{U(X)}{\eta_Z}{\gSC}} \\
	\MapC{U(Y)}{U(\nor{Z})}{\gSC} \ar[r] \ar[d]_\cong & \MapC{U(X)}{U(\nor{Z})}{\gSC} \ar[d]^\cong \\
	U(\MapC{Y}{\nor{Z}}{\pGSC}) \ar[r]_{U(\MapC{f}{\nor{Z}}{\pGSC})} & U(\MapC{X}{\nor{Z}}{\pGSC})
    }
	\]
	Since $f$ is a $JQ$-equivalence of normalized $\gSs$ by assumption and $\nor{Z}$ is a strictly unital coherently commutative monoidal quasi-category therefore
	the simplicial map $U(\MapC{f}{\nor{Z}}{\pGSC})$ is an equivalence of quasi-categories. The vertical arrows in the bottom rectangle are isomorphisms by corollary \ref{nor-map-sp}. The vertical arrows in the top rectangle are equivalences of quasi-categories because $U(X)$ and $U(Y)$ are cofibrant and $U(\nor{Z})$ is fibrant. Now the two-out-of-three property of weak equivalences in a model category tellas us that the top horizontal map in the above diagram namely $\MapC{U(f)}{\eta_Z}{\gSC}$ is a weak equivalence in the Joyal model category. By lemma \ref{char-CCME} we have shown that the map $U(f)$ is a $JQ$-equivalence.
	\end{proof}

We claim that $\nor{(-)}$ is a left adjoint of the forgetful functor $U:\pGSC \to \gSC$. The unit of this adjunction is given by the quotient map $\eta_X:X \to U(\nor{X})$. For a normalized $\gS$ $Y$ we have a canonical isomorphism (depicted by the dotted arrow) in the following diagram
\begin{equation}
\label{counit-map-isom}
\xymatrix{
	1 \ar[r] \ar@{=}[d] & U(Y) \ar[d] \ar@/^/[rdd]^{id} \\
	1 \ar[r] \ar@/_/[rrd] & U(\nor{U(Y)}) \ar@{-->}[rd]^{\epsilon_Y} \\
	&& U(Y)
}
\end{equation}
The diagram \eqref{counit-map-isom} is a composite arrow in the category $(1/\gSC)_\bullet$.
The image of the map $\epsilon_Y$ under the isomorphism from the remark following \eqref{nor-pointed-GS} gives us the counit map which we also denote by $\epsilon_Y$. 

The next proposition verifies our claim made above:
\begin{prop}
	The functor $\nor{(-)}:\gSC \to \pGSC$ is a left adjoint to the forgetful functor $U:\pGSC \to \gSC$.
\end{prop}
\begin{proof}
	We will prove this proposition by showing that the unit map $\eta_X$ constructed above is universal. Let $X$ be a $\gS$ and $Y$ be a normalized $\gS$ and
	$f:X \to U(Y)$ be a map in $\gSC$. We will show the extstence of a unique map
	$g:\nor{X} \to Y$ in $\pGSC$ such that the following diagram commutes in the category $\gSC$:
	\begin{equation}
	\label{universal-arrow}
	\xymatrix{
		X \ar[r]^{\eta_X \ \ \ \ } \ar[rd]_f & U(\nor{X}) \ar[d]^{U(g)} \\
		& U(Y)
	}
	\end{equation}
	The map $1 \to U(Y)$ in the diagram below is the image of the normalized $\gS$ $Y$
	under the isomorphism of categories in remark following \eqref{nor-pointed-GS}:
	\begin{equation*}
	\xymatrix{
		X[0] \ar[r] \ar[d] & X \ar[d]_{\eta_X} \ar@/^/[rdd]^f \\
		1 \ar[r] \ar@/_/[rrd] & U(\nor{X}) \ar@{-->}[rd]_{U(g)} \\
		&& U(Y)
	}
	\end{equation*}
	Since $U(Y)(0^+) = \ast$ therefore $f$ maps $X(0^+)$ to a point. This implies that the outer solid diagram in the figure above commutes. Since the square in the above diagram is a pushout square therefore there exists a unique (dotted) arrow which makes the whole diagram commutative.
	The lower commutative triangle in the diagram above is a map in the category $(1/\gSC)_\bullet$. The image of this map under the isomorphism of categories from the remark following \eqref{nor-pointed-GS} is a map $g:\nor{X} \to Y$ in $\pGSC$ whose image under the forgetful functor $U(g)$ makes the diagram \eqref{universal-arrow} commute.
\end{proof}
This proposition has the following consequence:
\begin{coro}
	\label{U-pres-cof}
	The forgetful functor $U:\pGSC \to \gSC$ maps $JQ$-cofibrations of normalized $\gSs$ to $JQ$-cofibrations.
\end{coro}
\begin{proof}
	Let $i:V \to W$ br a $JQ$-cofibration of normalized $\gSs$, we will show that
	$U(i)$ is a $JQ$-cofibration. By adjointness $U(i)$ is a cofibration if and only if $\nor{U(i)}$ is a $JQ$-cofibration of normalized $\gSs$. The following commutative square in $\pGSC$ shows that $U(i)$ is a cofibration because $i$ is one by assumption:
	\[
	\xymatrix{
		\nor{U(V)} \ar[r]_\cong^{\epsilon_V} \ar[d]_{\nor{U(i)}} & V \ar[d]^i \\
		\nor{U(W)} \ar[r]^\cong_{\epsilon_W}  & W  
	}
	\]
	
\end{proof}

Next we show that the adjunction $\nor{(-)})\dashv U$ is compatible with the model category structures \emph{i.e.} it is a Quillen adjunction.

\begin{lem}
	\label{Quillen-pair}
	The pair of adjoint functors $(\nor{(-)}), U)$ is a Quillen pair.
\end{lem}
\begin{proof}
	A pair of adjoint functors between two model categories is a Quillen pair if and only if the left adjoint preserves cofibrations and the right adjoint preserves fibrations between fibrant objects, see \cite[Prop. 7.15]{JT3}. Let $i:A \to B$
	be a cofibration in $\gSC$ and let $p:X \to Y$ be an acyclic fibration in $\pGSC$ then by proposition \ref{pres-acy-fib}, there is a lifting arrow in the following (outer) commutative diagram:
	\[
	\xymatrix{
	A \ar[r] \ar[d]_i & U(X) \ar[d]^{U(p)} \\
	B \ar[r] \ar@{-->}[ru] & U(Y)
    }
	\]
	By adjointness this lifting arrow exists if and only if there exists a lifting arrow in the following (adjunct) commutative diagram:
	\[
	\xymatrix{
		\nor{A} \ar[r] \ar[d]_{\nor{i}} & X \ar[d]^{p} \\
		\nor{B} \ar[r] \ar@{-->}[ru] & Y
	}
	\]
	Thus we have shown that for each cofibration $i$ in $\gSC$, its image $\nor{i}$ in $\pGSC$ has the left lifting property with respect to acyclic fibrations in the $JQ$-model category of normalized $\gSs$ $\pGSC$.
	Hence we have shown that the left adjoint preserves cofibrations.
	
	We recall that $JQ$-fibrations between $JQ$-fibrant normalized $\gSs$ are just strict $JQ$-fibrations of normalized $\gSs$. Now proposition \ref{pres-str-fib} tells us that $U$ preserves fibrations between fibrant normalized $\gSs$.
	Hence by  \cite[Prop. 7.15]{JT3} the adjunction in context is a Quillen pair.

\end{proof}

By definition, the counit map $\epsilon_Y:\nor{U(Y)} \to Y$ of the adjunction $(\nor{(-)}), U)$ is an isomorphism for each normalized $\gS$ $Y$. Now we want to show that the unit of of the same adjunction is a $JQ$-equivalence.
\begin{lem}
	The unit map $\eta_X:X \to U(\nor{X})$ is a $JQ$-equivalence for each $\gS$ $X$.
\end{lem}
\begin{proof}
	We have already seen in Proposition \ref{unit-eq-fib-obj} that this result holds when the $\gS$ $X$ is a coherently commutative monoidal quasi-category.
	Now we tackle the general case wherein $X$ is an arbitrary $\gS$. Since the unit map $\eta$ is a natural transformation therefore we have the following commutative diagram in the category $\gSC$:
	\begin{equation*}
	\xymatrix{
		X \ar[d]_{\eta_X} \ar[r] & R(X) \ar[d]^{\eta_{R(X)}} \\
		U(\nor{X})  \ar[r] & U(\nor{R(X)})
	}
	\end{equation*}
	where $X \to R(X)$ is a fibrant replacement of $X$ and therefore it is an acyclic $JQ$-cofibration and $R(X)$ is a coherently commutative monoidal quasi-category. Thus we have shown that the top and right vertical arrow in the commutative diagram above are $JQ$-equivalences. Now we want to show that the bottom horizontal arrow is also a $JQ$-equivalence. The functor $\nor{(-)}$ is a left Quillen functor, see \eqref{Quillen-pair}, therefore it preserves acyclic $JQ$-cofibrations. Now proposition \ref{U-pres-cof} says that $U$ preserves weak equivalences which implies that the bottom horizontal map is a $JQ$-equivalence.
	
\end{proof}
An easy consequence of the above lemma and the fact that the counit of the quillen pair $(\nor{(-)}, U)$ is a natural isomorphism is the following theorem:
\begin{thm}
	\label{Qui-eq-nor-unnor}
	The Quillen pair $(\nor{(-)}, U)$ is a Quillen equivalence.
	\end{thm}

%% file: EinMdlMarQCat.tex
\section[The model category of coherently commutative monoidal marked quasicategories]{The marked JQ-model category}

\label{EInf-Mar-QCat}
  The objective of this section is to construct a new model
  category structure on the category $\gSCM = [\gop; \sSetsM]$, where $\sSetsM$ is the model category of marked simplicial sets. We will refer to an object of $\gSCM$, namely a functor from $\gop$ to $\sSetsM$, as a \emph{marked} $\gS$. This new model category can be described as the model category of coherently commutative objects in $\sSetsM$. We begin by describing a \emph{projective} model category structure on $\gSCM$:
  
   \begin{df}
  	We call a map of marked $\gSs$
  	\begin{enumerate}
  		\item  A \emph{strict JQ-fibration} of marked $\gSs$ if it is degreewise a \emph{pseudo-fibration}  of 
  		marked simplicial sets \emph{i.e.} a fibration in the Joyal model category structure on marked simplicial sets.
  		\item A \emph{strict JQ-equivalence} of marked $\gSs$
  		if it is degreewise a \emph{categorical equivalence} of 
  		marked simplicial sets \emph{i.e.} a weak equivalence in the Joyal model category structure on marked simplicial sets.
  		\item A \emph{strict JQ-cofibration} of marked $\gSs$
  		if it has the left lifting property with respect to maps which are simultaneously strict $JQ$-fibrations and strict $JQ$-equivalences.
  	\end{enumerate}
  \end{df}
  
  \begin{thm}
  	\label{strict-cat-Q-M-model}
  	Strict JQ-equivalences, strict JQ-fibrations
  	and JQ-cofibrations of marked $\gSs$ provide the category $\gSCM$ with a combinatorial, left-proper model category structure.
  \end{thm}
  The model structure in the above theorem follows from \cite[Proposition A 3.3.2]{JL} and the left properness is a consequence of the left properness of the Joyal model category.
  \begin{nota}
  	We will refer to the above model category as the \emph{strict $JQ$-model category of marked simplicial sets}.
  	\end{nota}
  Next we will construct function objects for the above model category. For each pair $(F, K)$, where $F \in Ob(\gSCM)$ and $K \in Ob(\sSets)$,
  one can construct a $\gS$ which we denote by $\TensP{F}{K}{}$ and which is defined
  as follows:
  \[
  (\TensP{F}{K}{})(n^+) := F(n^+) \times \Fl{K},
  \]
  where the product on the right is taken in that category
  of simplicial sets. This construction is functorial in both variables.
  Thus we have a functor
  \begin{equation*}
  \label{mar-ten-over-sSets} 
  \TensP{-}-{}{}:\gSCM \times \sSets \to \gSCM. 
  \end{equation*}
  Now we will define a couple of function objects for the category $\gSCM$.
  The first function object enriches the category $\gSC$ over
  $\sSets$ \emph{i.e.} there is a bifunctor
  \[
  \MapC{-}{-}{\gSCM}:(\gSCM)^{op} \times \gSCM \to \sSets
  \]
  which assigns to each pair of objects $(X, Y) \in Ob(\gSCM) \times Ob(\gSCM)$, a simplicial set
  $\MapC{X}{Y}{\gSCM}$ which is defined in degree zero as follows:
  \[
  \MapC{X}{Y}{\gSCM}_0 := \gSCM(X, Y)
  \]
  and the simplicial set is defined in degree $n$ as follows:
  \begin{equation}
  \label{func-sp-mar-gS}
  \MapC{X}{Y}{\gSCM}_n := \gSCM(\TensP{X}{\Fl{\Delta[n]}}{}, Y)
  \end{equation}
  For any marked $\gS$ $X$, the functor $\TensP{X}{-}{}:\sSets \to \gSCM$ is
  left adjoint to the functor $\MapC{X}{-}{\gSCM}:\gSCM\to \sSets$. The \emph{counit} of this adjunction
  is the evaluation map $ev:\TensP{X}{\MapC{X}{Y}{\gSCM}}{} \to Y$
  and the \emph{unit} is the obvious simplicial map $K \to \MapC{X}{\TensP{X}{K}{}}{\gSCM}$.

  To each pair of objects $(K, X) \in Ob(\sSets) \times Ob(\gSCM)$ we can define a $\gS$ $\bHom{K}{X}{\gSCM}$, in degree $n$, as follows:
  \[
  (\bHom{K}{X}{\gSC})(n^+) := [\Fl{K}, X(n^+)] \ .
  \]
  This assignment
  is functorial in both variable and therefore we have a bifunctor
  \[
  \bHom{-}{-}{\gSCM}:\sSets^{op} \times \gSCM \to \gSCM.
  \]
  For any $\gS$ $X$, the functor $\bHom{-}{X}{\gSCM}:\sSets \to (\gSCM)^{op}$ is
  left adjoint to the functor $\MapC{-}{X}{\gSCM}:(\gSCM)^{op} \to \sSets$. 
  The following proposition summarizes the above discussion.
  \begin{prop}
  	\label{two-var-adj-cat-gcat-mar}
  	There is an adjunction of two variables
  	\begin{equation}
  	\label{two-var-adj-gcat-mar}
  	(\TensP{-}{-}{}, \bHom{-}{-}{\gSCM}, \MapC{-}{-}{\gSCM}) : \gSCM \times \sSets
  	\\  \to \gSCM.
  	\end{equation}
  	
  \end{prop}
  
  The following theorem follows from \cite[Remark A.3.3.4]{JL}. A direct proof can also be easily given by a straighforward verification of Lemma \ref{Q-bifunctor-char}$(2)$.
  \begin{thm}
  	\label{enrich-GamCAT-CAT-mar}
  	The strict model category of marked $\gSs$, $\gSCM$, is a $\sSetsQ$- model category with respect to the adjunction of two variables \eqref{two-var-adj-gcat-mar}.
  \end{thm}

   The adjoint functors $((-)^b, U)$, see \ref{Quil-eq-QCat-MQCat}, induce an adjunction
  \begin{equation}
  \label{ind-adj-mar-gSs}
  \Gamma(-)^b:\gSC \rightleftharpoons \gSCM:U.
  \end{equation}
  This adjunction is a Quillen equivalence in light of \cite[remark A 3.3.2]{JL} and \ref{Quil-eq-QCat-MQCat}:
  \begin{thm}
  	\label{str-mar-QCat-QE}
  	The adjoint pair $(\Gamma\Fl{(-)}, U)$ determines a Quillen equivalence between the strict $JQ$-model structure on $\gSC$ and the model category structure in theorem \ref{strict-cat-Q-M-model}.
  \end{thm}

The following three lemmas will be useful in proving various results in this section:
 \begin{lem}
 	\label{simp-adj}
 	For each pair $(X, Y)$ consisting of a $\gS$ $X$ and a marked $\gS$ $Y$, the above adjunction gives the following simplicial isomorphism:
 	\[
 	\MapC{\Gamma\Fl{(X)}}{Y}{\gSCM} \cong \MapC{X}{U(Y)}{\gSC}.
 	\]
 	\end{lem}
 \begin{proof}
 	By definition of the function space $\MapC{\Gamma\Fl{(X)}}{Y}{\gSCM}$, see \eqref{func-sp-mar-gS}, and \cite[Lemma 2.9]{GJ}  it is sufficient to observe that for each $n \in \Nat$
 	\[
 	\Gamma\Fl{(X)} \otimes \Fl{\Delta[n]} = \Gamma\Fl{(X \otimes \Delta[n])}.
 	\]
 	\end{proof}
 The next lemma is a consequence of the definition of the left adjoint functor $\Gamma\Fl{(-)}$:
 \begin{lem}
 	\label{left-adj-pres-str-eq}
 	The left adjoint functor $\Gamma\Fl{(-)}$ maps strict $JQ$-equivalences of $\gSs$ to strict $JQ$-equivalences of marked $\gSs$.
 	\end{lem}
 \begin{proof}
 	Let $F:X \to Y$ be a strict $JQ$-equivalence of $\gSs$. For each $n \in \Nat$ we have a categorical equivalence of simplicial sets $F(n^+):X(n^+) \to Y(n^+)$. We recall that each simplicial set is cofibrant in the Joyal model category therefore $F(n^+)$ is a weak equivalence between cofibrant objects.
 	
 	In degree $n$, the map of marked $\gSs$ $\Gamma\Fl{(F)}$ is the following map of marked simplicial sets:
 	\[
 	\Gamma\Fl{(F)}(n^+):\Fl{X(n^+)} \to \Fl{Y(n^+)}.
 	\]
 	In other words $\Gamma\Fl{(F)}(n^+) = \Fl{F(n^+)}$.
 	Since $\Fl{(-)}$ is a left adjoint of a Quillen equivalence therefore it preserves weak-equivalences between cofibrant objects. Thus $\Fl{F(n^+)} = \Gamma\Fl{(F)}(n^+)$ is a weak equivalence in the Joyal model category of marked simplicial sets for each $n^+ \in \gop$ and hence $\Gamma\Fl{(F)}$ is a strict $JQ$-equivalence of marked $\gSs$.
 	\end{proof}
 \begin{lem}
 	\label{coun-eq}
 	For any strict $JQ$-fibrant marked $\gSs$ $X$, the counit map $\epsilon_X:\Gamma\Fl{(X)} \to X$ is a strict $JQ$-equivalence of marked $\gSs$.
 	\end{lem}
 \begin{proof}
 	Since $(\Gamma\Fl{(-)}, U)$ is a Quillen pair therefore it induces a \emph{derived} adjunction $(\Gamma^L\Fl{(-)}, U^R)$ on the homotopy categories of the two model categories in context. For a strict $JQ$-fibrant $\gS$ $X$, the counit of this derived adjunction $\epsilon^D_X$ is defined as follows:
 	\begin{equation*}
 		\xymatrix{
 		\Gamma\Fl{(U(X))} \ar[r]^{\epsilon_X} & X \\
 		\Gamma\Fl{(Q(U(X)))} \ar[u]^{\Gamma\Fl{(q)}} \ar[ru]_{\epsilon^D_X}
 	   }
 		\end{equation*}
 		where $q:QU((X)) \to U(X)$ is a cofibrant replacement of $U(X)$ in the strict $JQ$-model category of $\gSs$.
 		By the previous lemma $\Gamma\Fl{(q)}$ is a strict $JQ$-equivalence of marked $\gSs$. Since the Quillen pair $(\Gamma\Fl{(-)}, U)$ is also a Quillen equivalence between the strict model categories therefore $\epsilon^D_X$ is a strict $JQ$-equivalence of marked $\gSs$. By the $2$ out of $3$ property of weak equivalences in a model category, we conclude that the counit map $\epsilon_X$ is a strict $JQ$-equivalence of marked simplicial sets.
 	\end{proof}

 Now we will construct another model category structure on $\gSCM$. The guiding principle of this new model structure is to endow its homotopy category with a semi-additive structure. In other words we want this new
 model category structure to have finite \emph{homotopy biproducts}.  
 We want to construct a left Bousfield localization of
 the strict $JQ$-model category of marked $\gSs$. For each pair $k^+, l^+ \in \gop$,
 we have the obvious \emph{projection maps} in $\gSC$
 \[
  \delta^{k+l}_k:(k+l)^+ \to k^+ \ \ \ \ and \ \ \ \ \delta^{k+l}_l:(k+l)^+ \to l^+.
 \]
 The maps
 \[
 \gop(\delta^{k+l}_k,-):\Gamma^{k} \to \Gamma^{k+l} \ \ \ \ and \ \ \ \ 
 \gop(\delta^{k+l}_l,-):\Gamma^{l} \to \Gamma^{k+l} 
 \]
 induce a map of $\gSs$ on the coproduct which we denote as follows:
 \[
  \Gamma\Fl{(h_k^l)}:\Gamma\Fl{(\Gamma^l)} \sqcup \Gamma\Fl{(\Gamma^l)} \to \Gamma\Fl{(\Gamma^{l+k})}.
 \]
 
 We now define a class of
 maps $\Gamma\E_\infty\S$ in $\gSCM$:
 \begin{equation*}
  \Gamma\E_\infty\S := \lbrace \Gamma\Fl{(h_k^l)}:\Gamma\Fl{(\Gamma^k)} \sqcup \Gamma\Fl{(\Gamma^l)} \to \Gamma\Fl{(\Gamma^{k+l})}:
  l, k \in \mathbb{Z}^+ \rbrace
 \end{equation*}
 \begin{df}
  We call a $\gS$ $X$ a $(\Delta \times \Gamma\E_\infty\S)$-\emph{local object}
  if it is a fibrant object in the strict $JQ$-model category and for each map $h_k^l \in \Gamma\E_\infty\S$, the induced simplicial map
  \begin{multline*}
  \HMapC{\Delta[n] \times \Gamma\Fl{(h_k^l)}}{X}{\gSCM}: \HMapC{\Delta[n] \times \Gamma\Fl{(\Gamma^{k+l})}}{X}{\gSCM} \to \\
  \HMapC{\Delta[n] \times (\Gamma\Fl{(\Gamma^k)} \sqcup \Gamma\Fl{(\Gamma^l)})}{X}{\gSCM},
  \end{multline*}
 is a homotopy equivalence of simplicial sets for all $n \ge 0$ where $\HMapC{-}{-}{\gSCM}$ is the simplicial function complexes associated with the strict $JQ$-model category $\gSCM$, see \cite{DK80}, \cite{DK1980} and \cite{DK3}.
 \end{df}
  Appendix \ref{Cat-Local} tell us that a model for $\HMapC{X}{Y}{\gSCM}$, where $X$ is cofibrant and $Y$ is fibrant, is the Kan complex
 $J(\MapC{X}{Y}{\gSCM})$, namely the maximal Kan complex contained in the quasicategory $\MapC{X}{Y}{\gSCM}$.

  The following proposition gives a characterization of
 $\E_\infty\S$-local objects
  \begin{prop}
  \label{char-mar-CCMC}
 \begin{sloppypar}
 A $\gS$ $X$ is a $(\Delta \times \Gamma\E_\infty\S)$-local object in $\gSCM$ if and only if its underlying $\gS$ $U(X)$ satisfies the Segal condition namely the functor
  \end{sloppypar}
 \begin{equation*}
 (U(X)(\partition{(k+l)}{k}), U(X)(\partition{(k+l)}{l})):U(X)((k+l)^+) \to U(X)(k^+) \times U(X)(l^+)
 \end{equation*}
 is a categorical equivalence of quasi-categories for all $k^+, l^+ \in \Ob(\gop)$.
 \end{prop}
 \begin{proof}
 We begin the proof by observing that each element of the set $\Gamma\E_\infty\S$ is a map of marked $\gSs$ between cofibrant marked $\gSs$. Theorem \ref{char-lo-QCat-en} and \ref{simp-adj} together imply that $X$ is a $(\Delta \times \Gamma\E_\infty\S)$-local object if and only if the following simplicial map
 \begin{equation*}
 \MapC{h^k_l}{U(X)}{\gSC}:\MapC{\gn{k+l}}{U(X)}{\gSC} \to \MapC{\gn{k} \sqcup \gn{l}}{U(X)}{\gSC}
 \end{equation*}
 is a categorical equivalence of quasi-categories.
 
 We observe that we have the following commutative square in $\sSetsQ$
 \begin{equation*}
 \xymatrix@C=24mm{
\MapC{\gn{(k+l)}}{U(X)}{\gSC}  \ar[d]_{\cong}\ar[r]^{ \MapC{h^k_l}{U(X)}{\gSC}} & \MapC{\gn{k} \sqcup \gn{l}}{U(X)}{\gSC}   \ar[d]^{\cong} \\
 U(X)((k+l)^+) \ar[r]_{(U(X)(\partition{k+l}{k}), U(X)(\partition{k+l}{l}))} & U(X)(k^+) \times U(X)(l^+)
 }
 \end{equation*}
 This implies that the simplicial map $(U(X)(\partition{k+l}{k}), U(X)(\partition{k+l}{l}))$ is a categorical equivalence of quasi-categories if and only if the functor $ \MapC{h^k_l}{U(X)}{\gSC}$ is one.
  \end{proof}

\begin{df}
 \label{CCMC-mar}
 We will refer to a $(\Delta \times \Gamma\E_\infty\S)$-local object as a \emph{coherently commutative monoidal marked quasi-category}.
 \end{df}
 
 \begin{df}
 A morphism of marked $\gSs$ $F:X \to Y$ is a $(\Delta \times \Gamma\E_\infty\S)$-local equivalence if for each coherently commutative monoidal marked quasi-category $Z$
 the following simplicial map
 \[
 \HMapC{F}{Z}{\gSCM}:\HMapC{Y}{Z}{\gSCM} \to \HMapC{X}{Z}{\gSCM}
 \]
 is a homotopy equivalence of simplicial sets.
 \end{df}
The following proposition follows from an argument similar to the one in the proof of proposition \ref{char-CCME}:
 \begin{prop}
  \label{char-mar-CCME}
 \begin{sloppypar}
 A morphism between two cofibrant marked $\gSs$ $F:X \to Y$ is an $(\Delta \times \Gamma\E_\infty\S)$-local equivalence if and only if the simplicial map  \end{sloppypar}
 \begin{equation*}
 \MapC{F}{Z}{\gSCM}:\MapC{Y}{Z}{\gSCM} \to \MapC{X}{Z}{\gSCM}
  \end{equation*}
 is an equivalence of quasi-categories for each coherently commutative monoidal marked quasi-category $Z$.
 \end{prop}

 \begin{df}
  We will refer to a $(\Delta \times \Gamma\E_\infty \S)$-local equivalence either as an \emph{equivalence of coherently commutative monoidal marked quasi-categories} or as a $JQ$-equivalence of marked $\gSs$.
 \end{df}

 We construct a
 new model category structure on the category $\gSCM$,
 by localizing the strict $JQ$-model category of marked $\gSs$ with respect to
 morphisms in the set $\Gamma\E_\infty\S$. 
\begin{thm}
 \label{loc-semi-add-mar}
 There is a left proper, combinatorial model category structure on the category of marked $\gSs$, $\gSCM$, in which
 \begin{enumerate}
 \item The class of cofibrations is the same as the class of
 JQ-cofibrations of marked $\gSs$.
 \item The weak equivalences are equivalences of coherently commutative monoidal marked quasi-categories.
 \end{enumerate}
 \begin{sloppypar}
 An object is fibrant in this model category if and only if it is a
  coherently commutative monoidal marked quasi-category. A fibration between two coherently commutative monoidal marked quasi-categories is a strict $JQ$-equivalence of marked simplicial sets.
   \end{sloppypar}
 \end{thm}
 \begin{proof}
 The strict model category of $\gSs$ is a combinatorial
 model category therefore the existence of the model structure
 follows from theorem \ref{local-tool}.
The last statement follows from $(1)$.
 \end{proof}
 \begin{nota}
 The model category constructed in theorem \ref{loc-semi-add-mar} will
 be referred to either as the model category of \emph{coherently commutative monoidal marked quasi-categories} or as the $JQ$-model category of marked $\gSs$.
 \end{nota}
\begin{thm}
	\label{Qu-adj-mar-gs}
	The adjoint pair $(\Gamma\Fl{(-)}, U)$ determines a Quillen adjunction between the $JQ$-model category marked $\gSs$ and the $JQ$-model category of $\gSs$.
\end{thm}
\begin{proof}
	As observed above, the adjunction \ref{adj-GS-MGS} determines a Quillen equivalence between the strict $JQ$-model category of marked $\gSs$ and the strict $JQ$-model category of $\gSs$. Since the two model categories in context are constructed as a left-Bousfield localization of the corresponding strict model categories therefore the cofibrations are the same as the corresponding strict model categories. This implies that the left adjoint $\Gamma\Fl{(-)}$ preserves cofibrations. We observe that the fibrations between fibrant objects, in the $JQ$-model category of marked $\gSs$, are the same as strict $JQ$-fibrations of marked $\gSs$. Now it follows from the aforementioned Quillen equivalence that the right adjoint $U$ preserves fibrations between fibrant objects. In light of proposition \cite[Prop. E.2.14.]{AJ1}, we conclude that $(\Gamma\Fl{(-)}, U)$ is a Quillen pair between the two $JQ$-model categories in context of the theorem.
\end{proof}

\begin{prop}
	\label{Rt-adj-pres-Eq}
	The right adjoint functor $U$ maps $JQ$-equivalences of marked $\gSs$ to $JQ$-equivalences of $\gSs$.
\end{prop}
\begin{proof}
	Let us first assume that $F:X \to Y$ is a $JQ$-equivalence of marked $\gSs$ between strict $JQ$-fibrant and cofibrant marked $\gSs$, \emph{i.e.} $X$ and $Y$ are both cofibrant and fibrant objects is the strict $JQ$-model category of marked $\gSs$. By lemma \ref{coun-eq} we have the following commutative square:
	\begin{equation*}
	\xymatrix{
		\Gamma\Fl{(U(X))} \ar[r]^{\Gamma\Fl{(F)}} \ar[d]_{\epsilon_X} & \Gamma\Fl{(U(Y))} \ar[d]^{\epsilon_Y} \\
		X  \ar[r]_F & Y 
	}
	\end{equation*}
	wherein the vertical maps are strict $JQ$-equivalences of marked $\gSs$. This implies that $\Gamma\Fl{(F)}$ is a $JQ$-equivalence whenever $F$ is one. Now for each coherently commutative monoidal quasi-category $Z$, we have the following commutative diagram of maps between mapping spaces in $\sSets$:
	\begin{equation*}
	\xymatrix@C=32mm{
		\MapC{\Gamma\Fl{(U(Y))}}{\Gamma\Fl{(Z)}}{\gSCM} \ar[r]^{\MapC{\Gamma\Fl{(U(F))}}{\Gamma\Fl{(Z)}}{\gSCM}} \ar[d]_{\cong} & \MapC{\Gamma\Fl{(U(X))}}{\Gamma\Fl{(Z)}}{\gSCM} \ar[d]^{\cong} \\
		\MapC{U(Y)}{U(\Gamma\Fl{(Z)})}{\gSC}   \ar[r]_{\MapC{U(F)}{U(\Gamma\Fl{(Z)})}{\gSC}} \ar@{=}[d]  & \MapC{U(X)}{U(\Gamma\Fl{(Z)})}{\gSC} \ar@{=}[d] \\
		\MapC{U(Y)}{Z}{\gSC}   \ar[r]_{\MapC{U(F)}{Z}{\gSC}}  & \MapC{U(X)}{Z}{\gSC}
	}
	\end{equation*}
	where the vertical isomorphisms follow from lemma \ref{simp-adj} and the vertical equalities follow from the observation that the unit map of the adjunction in context is the identity. The above commutative diagram of mapping spaces implies that $U(F)$ is a $JQ$-equivalence of $\gSs$ whenever $F$ is a $JQ$-equivalence of marked $\gSs$.
	
	Let $F:X \to Y$ be a $JQ$-equivalence of marked $\gSs$. By \cite[Prop. 8.1.17]{Hirchhorn} we can choose a   cofibrant- fibrant replacement functor $(R, r)$ in the strict $JQ$-model category structure on marked $\gSs$. This functor gives us the following commutative square:
	\begin{equation*}
	\xymatrix{
	R(X) \ar[r]^{R(F)} & R(Y) \\
	X \ar[u]^{r_X} \ar[r]_F & Y \ar[u]_{r_Y}
  }
	\end{equation*}
	The vertical maps are strict $JQ$-equivalences (acyclic cofibrations) of marked $\gSs$ and the top horizontal arrow is a $JQ$-equivalence of marked $\gSs$ between objects which are both cofibrant and fibrant in the strict $JQ$-model category of marked $\gSs$. It is easy to see that $U$ maps strict $JQ$-equivalences of marked $\gSs$ to strict $JQ$-equivalences of $\gSs$. Further, $U(R(F))$ is a (strict) $JQ$-equivalences of $\gSs$ from the arguments made earlier in the proof. Now the $2$ out of $3$ property of weak-equivalences in a model category tells us that $U(F)$ is a $JQ$-equivalences of $\gSs$.
\end{proof}

Now we state and prove the main result of this section:
\begin{thm}
	\label{Qu-eq-mar-ss}
 The Quillen pair of theorem \eqref{Qu-adj-mar-gs} is a Quillen equivalence.
 \end{thm}
 \begin{proof}
Let $X$ be a $Q$-cofibrant $\gS$ and $Y$ be a $JQ$-fibrant marked $\gS$. We will show that a map $F:\Gamma\Fl{(X)} \to Y$ is a $JQ$-equivalence of marked $\gSs$ if and only if its adjunct map $\phi(F):X \to U(Y)$ is a $JQ$-equivalence of $\gSs$. 
We consider the following commutative diagram provided by the adjunction $(\Gamma\Fl{(-)}, U)$
\begin{equation*}
\xymatrix{
U(\Gamma\Fl{(X)}) \ar[r]^{F} & U(Y) \\
X \ar[u]^{\eta_X} \ar[ru]_{\phi(F)}
}
\end{equation*}
where $\eta_X$ is the unit map. Now the result follows from this diagram, proposition \ref{Rt-adj-pres-Eq} and the observation that the unit $\eta_X$ is the identity map.

%
 	\end{proof}
 
 The next lemma is a consequence of the fact that the left adjoint functor $\Gamma\Fl{(-)}$ preserves fibrant objects, which follows easily from the above charaterization of coherently commutative monoidal marked $\gSs$. The following lemma tells us a strong property of the left adjoint $\Gamma\Fl{(-)}$ which is not posessed by every left Quillen functor of a Quillen adjunction:
 \begin{lem}
 	\label{left-adj-ref-pre}
 	The left adjoint functor $\Gamma\Fl{(-)}$ preserves weak-equivalences namely the $JQ$-equivalences. Further, the left adjoint $\Gamma\Fl{(-)}$ reflects weak equivalences whose codomain is fibrant. 
 \end{lem}
 \begin{proof}
 	Let $F:X \to Y$ be a morphism of $\gSs$ then it can be factored as follows:
 	\begin{equation*}
 	\xymatrix{
 		& P(F) \ar[rd]^{p(F)} \\
 		X \ar[ru]^{i(F)} \ar[rr]_F && Y
 	}
 	\end{equation*}
 	where $i(F)$ is an acyclic $JQ$-cofibration and $p(F)$ is a $JQ$-fibration of $\gSs$. Let us first assume that $F$ is a $JQ$-equivalence of $\gSs$. Now the $2$ out of $3$ property of weak-equivalences in a model category tells us that $p(F)$ is a acyclic $JQ$-fibration which is the same as a strict acyclic $JQ$-fibration. Theorem \ref{Qu-adj-mar-gs} tells us that $\Gamma\Fl{(-)}$ is a left Quillen functor and therefore it preserves acyclic $JQ$-cofibrations. By lemma \ref{left-adj-pres-str-eq} it also preserves strict $JQ$-equivalences. In other words both $\Gamma\Fl{(i(F))}$ and $\Gamma\Fl{(p(F))}$ are $JQ$-equivalences of marked $\gSs$. Therefore their composite $\Gamma\Fl{(F)} = \Gamma\Fl{(i(F))} \circ \Gamma\Fl{(p(F))}$ is also a $JQ$-equivalence. Hence we have shown that the left adjoint preserves all weak-equivalences.
 	
 	Now let us assume that $Y$ is a $JQ$-fibrant marked $\gS$. Let us further assume that $\Gamma\Fl{(F)}:\Gamma\Fl{(X)} \to \Gamma\Fl{(Y)}$ is a $JQ$-equivalence of marked $\gSs$. Applying the left adjoint functor $\Gamma\Fl{(-)}$ to the above factorization of $F$ gives us the equality $\Gamma\Fl{(F)} = \Gamma\Fl{(i(F))} \circ \Gamma\Fl{(p(F))}$. Since $\Gamma\Fl{(-)}$ is a left Quillen functor therefore $\Gamma\Fl{(i(F))} $ is an acyclic $JQ$-cofibrations of marked $\gSs$. By assumption $\Gamma\Fl{(F)}$ is a weak-equivalence, therefore $\Gamma\Fl{(p(F))}$ is a $JQ$-equivalence of marked $\gSs$. By proposition \ref{char-mar-CCMC}, both $\Gamma\Fl{(P(F))}$ and $\Gamma\Fl{(Y)}$ are $JQ$-fibrant marked $\gSs$. The right adjoint $U$ is a right Quillen functor therefore it preserves weak-equivalences between fibrant objects, hence $U(\Gamma\Fl{(p(F))}) = p(F)$ is a $JQ$-equivalence of marked $\gSs$. Since $i(F)$ is an acyclic $JQ$-cofibration of $\gSs$ therefore $F = p(F) \circ i(F)$ is a $JQ$-equivalence of $\gSs$.

 \end{proof}

%% file: CoCartFibModel.tex
\section[Comparison with Symmetric monoidal quasi-categories]{Comparison with Symmetric monoidal quasi-categories}
\label{cocart-fib-model}

In this section we compare our functor based model for coherently commutative monoidal quasi-categories with a recent coCartesian fibration based model for similar objects which have been named \emph{symmetric monoidal} quasi-categories. We recall the following definition from \cite{JL2}:
\begin{df}
	\label{SM-QCat}
	A \emph{symmetric monoidal} quasi-category is a coCartesian fibration $p:X \to \nGop$ such that for each pair of objects $k^+, l^+ \in \gop$ the projection maps $\partition{k+l}{k}:(k+l)^+ \to k^+$ and $\partition{k+l}{l}:(k+l)^+ \to l^+$ induce morphisms of quasi-categories on the fibers
	\[
	 X((k+l)^+) \to X(k^+) \ \ \ \  \textit{and} \ \ \ \ X((k+l)^+) \to X(l^+)
	 \]
	 
	  which determine a categorical equivalence $X((k+l)^+) \to X(k^+) \times X(l^+)$.

	\end{df}
The main result of this section implies that the underlying coCartesian fibrations of symmetric monoidal quasi-categories can be \emph{rectified} (upto  equivalence) into an honest functor.
We recall that a coCartesian fibration $p:X \to \nGop$ can be viewed as an object $(\Nt{X}, p)$ in $\sSetsMG$ wherein the  marked edges of $\Nt{X}$ are exactly the $p$-coCartesian edges. An object is fibrant in the coCartesian model category $\ccMdl{\nGop}$ if and only if it is isomorphic to some $(\Nt{X}, p)$. In this section we will construct another model category structure on $\sSetsMG$, denoted $\sSetsMGSM$, in which an object is fibrant if and only if it is isomorphic to some  $(\Nt{X}, p)$ whose underlying coCartesian fibration represents a symmetric monoidal quasi-category. The main result of this section is to show that the relative nerve functor defined in \cite[Sec. 3.2.5]{JL} is a right Quillen functor of a Quillen equivalence between the $JQ$-model category of marked $\gSs$ $\gSCM$ and the model category $\sSetsMGSM$.

For each $k^+ \in \gop$ we define a simplicial set $\NElG{k}$ which is the nerve of the overcategory $\ovCatGen{k^+}{ \gop}$. This simplicial set is a quasi-category and is equipped with an obvious projection map
\[
p:\NElG{k} \to N(\gop).
\]
This map is a pseudo-fibration \emph{i.e.} a fibration in the Joyal model category. We regard this projection map as a morphism of marked simplicial sets as follows:
\[
p:\Sh{\NElG{k}}\to \Sh{N(\gop)}
\]
\begin{rem}
	We observe that $\mRN{\gn{k}} \cong \NElG{k}$.
	We will denote the quasi-category $\mRN{\gn{k}}$ by $\NElG{k}$. 
\end{rem}
\begin{nota}
	We will denote the value of the functor $\mRN{\bullet}$ on a marked $\gS$ $X$ either as $\mRN{X}$ or as $\mRN{\bullet}(X)$.
	\end{nota}
For each pair of objects $k^+, l^+ \in \gop$ we can define a map in the overcategory $\sSetsMG$ by the following commutative triangle:
\begin{equation*}
\xymatrix{
\Sh{\NElG{k}} \sqcup \Sh{\NElG{k}} \ar[rd] \ar[rr] &&  \Sh{\NElG{(k + l)}} \ar[ld] \\
& \Sh{N(\gop)}
}
\end{equation*}
where the diagonal maps are the obvious projections.
We denote the above morphism by $\Upsilon(k, l)$
We define a set of maps 
\[
\Upsilon\S = \lbrace \Upsilon(k, l): k, l \in \Nat \rbrace
\]
The category $\sSetsMG$ is tensored over $\sSets$ therefore we can define the following set of maps:
\begin{equation}
\Delta \times \Upsilon\S = \lbrace \Upsilon(k, l) \otimes \Delta[n]: n,k, l \in \Nat \rbrace
\end{equation}

 The following proposition is an easy consequence of the enrichment of the coCartesian model category $\sSetsMG$ over $\sSetsQ$ and the main result of appendix \ref{Cat-Local}.
\begin{prop}
	A coCartesian fibration $p:X \to N(\gop)$ viewed as an object of $\sSetsMG$, namely $(\Nt{X}, p)$,  is a $(\Delta \times \Upsilon\S)$-local object if and only if the following simplicial morphism of mapping spaces is a categorical equivalence of quasi-categories:
	\begin{multline*}
	\Flmap{\Upsilon(k, l)}{\Nt{X}}{\gop}:\Flmap{\NElG{(k + l)}}{\Nt{X}}{\gop} \to
	\Flmap{\NElG{k} \sqcup \NElG{l}}{\Nt{X}}{\gop} \cong \\ \Flmap{\NElG{k}}{\Nt{X}}{\gop} \times \Flmap{\NElG{l}}{\Nt{X}}{\gop}
	\end{multline*}
	for each pair $k^+, l^+ \in \gop$.
	\end{prop}

\begin{nota}
	\label{fibers}
	The fiber of a coCartesian fibration $X \to S$ over $s \in S_0$ will be denoted by $X(s)$.
	\end{nota}
\begin{rem}
	\label{Eq-on-fibers}
	Each fibrant object in $\ccMdl{N(\gop)}$ is a coCartesian fibration.
	The above proposition encodes the idea that for each pair of objects $k^+, l^+ \in \gop$, the fiber over $(k+l)^+$, $X((k+l)^+)$ is equivalent (as quasi-categories) to the product of fibers over $k^+$ and $l^+$, $X(k^+) \times X(l^+)$.  We recall from \cite[Lemma 3.12]{HGC} that there is a categorical equivalence
	\[
	\Flmap{\NElG{k}}{\Nt{X}}{\gop} \overset{\Flmap{id_{k^+}}{\Nt{X}}{\gop}}\to \Flmap{\Delta[0]}{\Nt{X}}{\gop} \cong U(X(k^+))
	\]
	for each $k^+ \in \gop$.
	\end{rem}

%
%

 Next we define another model category structure on the overcategory $\sSetsMG$:
 
 \begin{thm}
 	\label{mdl-cat-SM-QCat}
 	There is a left proper, combinatorial model category structure on the category $\sSetsMG$, in which a map is a
 	\begin{enumerate}
 		\item cofibrations if it is a cofibration in $\ccMdl{\nGop}$ namely its underlying simplicial map is a monomorphism.
 		\item weak equivalences if it is a $(\Delta \times \Upsilon\S)$-local equivalence.
 		\item fibration if it has the right lifting property with respect to all maps which are simultaneously cofibrations and weak-equivalences.
 	\end{enumerate}
 	\end{thm}
 \begin{proof}
 	The desired model category is obtained by a left Bousfield localization of the coCartesian model category $\ccMdl{\nGop}$.
 	The model category structure follows from theorem \ref{local-tool} with the set of generators for the localization being the set $\Delta \times \Upsilon\S$.

 	\end{proof}
 
 \begin{nota}
 	\begin{sloppypar}
 	We will refer to the above model category as the model category of \emph{symmetric monoidal quasi-categories} and denote this model category by $\sSetsMGSM$.
 	\end{sloppypar}
 \end{nota}

\begin{thm}
	The adjoint pair $(\mRN{\bullet}, \mRNL{\bullet})$ is a Quillen pair between the model category of symmetric monoidal quasi-categories and the $JQ$-model category of marked $\gSs$.
\end{thm}
\begin{proof}
	In light of \cite[Prop. E.2.14]{AJ1}, it is sufficient to show that the left adjoint $\mRNL{\bullet}$ maps cofibrations to cofibrations and the right adjoint $\mRN{\bullet}$ maps fibrations between fibrant objects to fibrations. We recall from \cite[Prop. 3.2.5.18(2)]{JL} that the adjoint pair $(\mRNL{\bullet}, \mRN{\bullet})$ is a Quillen pair with respect to the coCartesian model category structure $\ccMdl{N(\gop)}$ and the strict $JQ$-model category structure on $\gSCM$. Since the two model category structures in context are left Bousfield localizations, which preserves cofibrations, therefore the left adjoint $\mRNL{\bullet}$ will (still) peserve cofibrations. The fibrations between fibrant objects in the $JQ$-model category are  strict $JQ$-fibrations, therefore the right adjoint $\mRN{\bullet}$ will map such a fibration to a fibration in  $\sSetsMGSM$. Further the right adjoint functor $\mRN{\bullet}$ maps fibrant objects in the $JQ$-model category to fibrant objects in $\sSetsMGSM$. This implies that the right adjoint preserves fibrations between fibrant objects.

\end{proof}

\begin{df}
	An object $Z$ in $\sSetsMG$ is called a \emph{local} object if the following composite map:
	\[
	Z \overset{\eta_Z} \to \mRN{\bullet}\left( \mRNL{\bullet}(Z) \right) \overset{r} \to \mRN{\bullet}\left( R\left( \mRNL{\bullet}(Z) \right) \right)
	\]
	is a weak-equivalence in $\ccMdl{\nGop}$, where $(R, r)$ is a fibrant replacement replacement functor in the $JQ$-model category and $\eta_Z$ is the unit map.
	\end{df}
\begin{rem}
	\label{inv-und-cc-Eq}
	The notion of a local object is invariant under coCartesian equivalences.
	\end{rem}
The following lemma will be useful in writing the proof of the main theorem of this section:
\begin{lem}
	\label{JQ-fib-cc-local}
	Each fibrant object $Z$ in $\sSetsMGSM$ is a local object.
	\end{lem}
\begin{proof}
	In light of \cite[Prop. 3.1.4.1]{JL} we may assume that the underlying simplicial map $U(p):U(Z) \to U(\Sh{\nGop})$ is a coCartesian fibrations and the marked edges of $Z$ are the $p$-coCartesian edges \emph{i.e.}. $Z = \Nt{U(Z)}$.
	We begin by making the observation that for each $k^+ \in \gop$, the marked simplicial set $\mRNL{Z}(k^+)$ is equivalent to the fiber over $k^+$ of $Y$. Since $Z$ is fibrant therefore $\mRNL{Z}$ satisfies the Segal condition. Now a fibrant replacement in the strict $JQ$-model category will produce a coherently commutative monoidal marked quasi-category. It follows that any fibrant replacement of $\mRNL{Z}$ in the $JQ$-model category is a fibrant replacement in the strict $JQ$-model category. This gives us the following map in $\sSetsMG$:
	\[
	Z \overset{\eta_Z} \to \mRN{\bullet}\left( \mRNL{\bullet}(Z) \right) \overset{r} \to \mRN{\bullet}\left( R\left( \mRNL{\bullet}(Z) \right) \right).
	\]
	We recall that for each marked $\gS$ $X$, the fiber over each $k^+ \in \gop$ of $p:\mRN{X} \to \nGop$ is isomorphic to $X(k^+)$.
	This implies that the above map is a \emph{pointwise} equivalence \emph{i.e.} for each $k^+ \in \gop$ the above map induces a categorical equivalence of (marked) simplicial sets on the fiber over $k^+$. Since both $Z = \Nt{U(Z)}$ and $\mRN{\bullet}\left( R\left( \mRNL{\bullet}(Z) \right) \right) = \Nt{U(\mRN{\bullet}\left( R\left( \mRNL{\bullet}(Z) \right) \right))}$, it follows from \cite[Prop. 3.3.1.5]{JL} that the above map is a coCartesian equivalence. Now remark \ref{inv-und-cc-Eq} implies that $Z$ is a local object. 
	\end{proof}

Now we state and prove the main result of this section:
\begin{thm}
	\label{Eq-SM-CCM-Qcat}
	The Quillen pair $(\mRNL{\bullet},\mRN{\bullet})$ is a Quillen equivalence between the model category of symmetric monoidal quasi-categories and the $JQ$-model category of marked $\gSs$.
\end{thm}
\begin{proof}
	We will prove this theorem by verifying \cite[Prop. 1.3.13(b)]{Hovey}. We choose a fibrant replacement functor $(R, r)$ in the $JQ$-model category of marked $\gSs$. We will first show that the following composite map
	\[
	X \overset{\eta_X} \to \mRN{\bullet}\left(\mRNL{\bullet}(X)\right) \to \mRN{\bullet}(R\left(\mRNL{\bullet}(X)\right))
	\]
	is a weak equivalence in $\sSetsMGSM$, for each cofibrant object $X$ in $\sSetsMGSM$. We choose another fibrant replacement functor $(R^\otimes, r^\otimes)$ in $\sSetsMGSM$. Now we have the following commutative diagram in $\gSCM$:
	\begin{equation*}
	\xymatrix{
	X \ar[r]^{\eta_X \ \ \ \ \ \ \ } \ar[d]_{r^\otimes_X} &  \mRN{\bullet}\left(\mRNL{\bullet}(X)\right) \ar[r]^A \ar[d] & \mRN{\bullet}(R\left(\mRNL{\bullet}(X)\right)) \ar[d]^C  \\
	R^\otimes(X) \ar[r]_{\eta_{R^\otimes(X)} \ \ \ \ \ \ \ \ \ \ \ \ \  } & \mRN{\bullet}\left(\mRNL{\bullet}(R^\otimes(X))\right) \ar[r]_{B \ \ \ \ } & \mRN{\bullet}(R\left(\mRNL{\bullet}(R^\otimes(X))\right))
     }
	\end{equation*}
	where $A$ is the map $\mRN{\bullet}\left(r_{\left(\mRNL{\bullet}(X)\right)}\right)$, $B$ is the map $\mRN{\bullet}\left(r_{\left(\mRNL{\bullet}(R^\otimes(X))\right)}\right)$ and the downward map $C$ is $\mRN{\bullet}\left(R\left(\mRNL{\bullet}(r^\otimes(X))\right)\right)$.
	Since the object $R^\otimes(X)$ is a fibrant object in $\sSetsMGSM$, it follows from lemma \ref{JQ-fib-cc-local} that the bottom row of the above diagram is a coCartesian equivalence.
	Since $r^\otimes_X$ is an acyclic cofibration therefore the left Quillen functor $\mRNL{\bullet}$ preserves it. Thus $R(\mRNL{\bullet}(r^\otimes(X)))$ is a weak-equivalence between fibrant objects which the right Quillen functor $\mRN{\bullet}$ will preserve. Thus the rightmost vertical arrow is also a weak-equivalence in $\sSetsMGSM$. Now the $2$ out of $3$ property of weak-equivalences implies that the top row of the above diagram is a weak-equivalence in $\sSetsMGSM$.
	
	Next we choose a cofibrant replacement functor $(Q, q)$ in $\sSetsMGSM$. We will show that the following map is a weak-equivalence for each coherently commutative monoidal marked quasi-category $Y$:
	\begin{equation*}
	\mRNL{\bullet}(Q\left(\mRN{\bullet}(Y)\right)) \overset{G} \to \mRNL{\gop}(\mRN{\bullet}(Y)) \overset{\epsilon_Y}\to Y
	\end{equation*}
	where $G$ is the map $\mRNL{\bullet}(q_{\mRN{\bullet}(Y)})$. The Quillen equivalence \cite[Prop. 3.2.5.18(2)]{JL} implies that for each fibrant object $Y$, the counit map $\epsilon_Y$ is a coCartesian equivalence and hence a weak-equivalence in $\sSetsMGSM$. Since $q_{\mRN{\bullet}(Y)}$ is a weak-equivalence between fibrant objects in $\sSetsMGSM$ therefore it is a coCartesian equivalence which will be preserved by the left Quillen functor $\mRNL{\bullet}$. Thus we have shown that the above composite map is a weak equivalence in $\sSetsMGSM$.
	
	\end{proof}
Finally we give another characterization for symmetric monoidal quasi-categories. We recall a right Quillen functor $\mRNR:\sSetsMG \to \gSCM$ defined in the paper \cite{HGC}:
\[
\mRNR(X)(k^+) = [\NElG{k}, X]^+_{\gop}.
\]
The following corollary provides another characterization of fibrant objects. It is an easy consequence of the above theorem:

\begin{coro}
	\label{char-LO-objs}
	An coCartesian fibration $p:X \to N(\gop)$ viewed as an object of $\sSetsMGSM$ is a fibrant object if and only if the marked $\gS$ $\mRNR(X)$ is a coherently commutative monoidal marked quasi-category.
\end{coro}

%% file: EnrichmentOfGS.tex
\section{Quillen Bifunctors}
\label{en-of-GS}
  The objective of this section is to recall the notion
  of \emph{Quillen Bifunctors}. In order to do so, we
  begin with the definition of a \emph{two variable adjunction}:
 
 \begin{df}
 \label{2-var-adjunction}
 Suppose $\C$, $\D$ and $\E$ are categories. An \emph{adjunction
 of two variables} from $\C \times \D$ to $\E$ is a quintuple
 $\left(\otimes, \bhom_\C, \map_\C, \phi, \psi \right)$, where
 \[
  \otimes:\C \times \D \to \E, \ \ \ \ \bhom_\C:\D^{op} \times \E \to \C,
  \ \text{and} \ \ \ \ \map_\C:\C^{op} \times \E \to \D
 \]
 are functors and $\phi$, $\psi$ are the following natural transformations
 \[
  \C(C, \bhom_\C(D, E)) \overset{\phi^{-1}}{\underset{\cong} \to} \E(\C \otimes D, E)
  \overset{\psi}{\underset{\cong} \to} \D(D, \map_\C(C, E)).
 \]
  \end{df}
 The following definition is based on Quillen's $SM7$ axiom,
 see \cite{Quillen}.
 
 \begin{df}
  \label{Q-adj-2-var}
  Given model categories $\C$, $\D$ and $\E$, an adjunction
  of two variables, $\left(\otimes, \bhom_\C, \map_\C, \phi, \psi \right):
  \C \times \D \to \E$, is called a \emph{Quillen adjunction of two variables}, if, given a
  cofibration $f:U \to V$ in $\C$ and a cofibration $g:W \to X$ in $\D$,
  the induced map
  \[
   f \Box g:(V \otimes W) \underset{U \otimes W} \coprod (U \otimes X) \to V \otimes X
  \]
  is a cofibration in $\E$ that is trivial if either $f$ or $g$ is.
  We will refer to the left adjoint of a Quillen adjunction of two
  variables as a \emph{Quillen bifunctor}.

 \end{df}
 The following lemma provides three equivalent characterizations
 of the notion of a Quillen bifunctor. These will be useful in this paper
 in establishing enriched model category structures.
 \begin{lem}\cite[Lemma 4.2.2]{Hovey}
  \label{Q-bifunctor-char}
   Given model categories $\C$, $\D$ and $\E$, an adjunction
  of two variables, $\left(\otimes, \bhom_\C, \map_\C, \phi, \psi \right):
  \C \times \D \to \E$. Then the following conditions are equivalent:
  \begin{enumerate}
   \item [(1)] $\otimes:\C \times \D \to \E$ is a Quillen bifunctor.
   
   \item[(2)] Given a cofibration $g:W \to X$ in $\D$ and a fibration
   $p:Y \to Z$ in $\E$, the induced map
   \[
    \bhom_\C^{\Box}(g, p):\bhom_\C(X, Y) \to \bhom_\C(X, Z)
    \underset{\bhom_\C(W, Z)}\times \bhom_\C(W, Y)
   \]
   is a fibration in $\C$ that is trivial if either $g$ or $p$ is a
   weak equivalence in their respective model categories.
   
   \item[(3)] Given a cofibration $f:U \to V$ in $\C$ and a fibration
   $p:Y \to Z$ in $\E$, the induced map
   \[
    \map_\C^{\Box}(f, p):\map_\C(V, Y) \to \map_\C(V, Z) \underset{\map_\C(W, Z)}\times \map_\C(W, Y)
   \]
   is a fibration in $\C$ that is trivial if either $f$ or $p$ is a
   weak equivalence in their respective model categories.

  \end{enumerate}

 \end{lem}

%% file: Cat-Localization.tex
\section{On local objects in a model category enriched over quasicategories}
\label{Cat-Local}
A very detailed sketch of this appendix was provided to the author by Andre Joyal. This appendix contains some key results which have made this research possible.
\subsection{Introduction}
A model category $E$ is enriched over quasi-categories if the category $E$ is simplicial, tensored and cotensored, and the functor $[- ; -]: E^{op} \times$E$\to \sSets$ is a Quillen functor of two variables, where $\sSets = (\sSets, Qcat)$ is the model structure for quasi-categories. The purpose of this appendix is to introduce the notion of local object with respect to a map in a model category enriched over quasi-categories.
\subsection{Preliminaries}
Recall that a Quillen model structure on a category $E$ is determined by its class of cofibrations together with its class of fibrant objects. For examples, the category of simplicial sets $\sSets = [\Delta^{op},Set]$ admits two model structures in which the cofibrations are the monomorphisms: the fibrant objects are the Kan complexes in one, and they are the quasi-categories in the other. We call the former the model structure for Kan complexes and the latter the model structure for quasi-categories. We shall denote them respectively by $(\sSets, Kan)$ and $(\sSets, QCat)$.
Recall that a simplicial category is a category enriched over simplicial sets. There is a notion of simplicial functor between simplicial categories, and a notion of strong natural transformation between simplicial functors. If $E = (E, [- ; -])$ is a simplicial category, then so is the category $\SFunc{E}{\sSets}$ of simplicial functors $E \to \sSets$. A simplicial functor $F :$E$\to \sSets$ isomorphic to a simplicial functor $[A, -] :$E$\to \sSets$ is said to be \emph{representable}. Recall Yoneda lemma for simplicial functors: if $F :$E$\to \sSets$ is a simplicial functor and $A \in E$, then the map $y : Nat([A,-],F) \to F(A)_0$ defined by putting $y(\alpha) = \alpha (A)(\unit{A})$ for a strong natural transformation $\alpha : [A, -] \to F$ is bijective. The simplicial functor $F$ is said to be represented by a pair $(A, a)$, with $a \in F(A)_0$ , if the unique strong natural transformation $\alpha : [A, -] \to F$ such that $\alpha(A) (\unit{A} ) = a$ is invertible. We say that a simplicial category $E = (E,[-,-])$ is \emph{tensored by} $\Delta$ if the simplicial functor
\[
[A, -]^{\Delta[n]} :E \to \sSets
\]
is representable (by an object denoted $\Delta[n] \times A$) for every object $A \in E$ and every $n \ge 0$. If $E$
has finite colimits and is tensored by $\Delta$, then it is tensored by finite simplicial sets: the simplicial
 functor
is representable (by an object $K \times A$) for every object $A \in E$ and every finite simplicial set $K$.
Dually, we say that a simplicial category $E$ is \emph{cotensored by} $\Delta$ if the simplicial functor 
\[
[-, X]^{\Delta[n]} : E^{op} \to \sSets
\]
is representable (by an object denoted $X^{ \Delta[n]}$) for every object $X \in E$ and every $n \ge 0$. If $E$ has finite limits and is cotensored by $\Delta$, then it is cotensored by finite simplicial sets: the simplicial functor
\[
[-,X]^K :E^{op} \to \sSets
\]
is representable by an object $X^K$ for every object $X \in E$ and every finite simplicial set $K$.
Recall that a model category $E$ is said to be simplicial if the category $E$ is simplicial, tensored and cotensored by 
$\Delta$ and the functor $[-, -] :E^{op} \times E \to \sSets$ is a Quillen functor of two variables, where $\sSets = (sSet,Kan)$. The last condition implies that if $A \in E$ is cofibrant and $X \in E$ is fibrant, then the simplicial set $[A,X]$ is a Kan complex. For this reason, we shall say that a simplicial model category is enriched over Kan complexes.
\begin{df}
\label{enrich-QCat}
We shall say that a model category $E$ is enriched over quasi-categories if the category $E$ is simplicial, tensored and cotensored over $\Delta$ and the functor $[-,-] : E^{op} \times E \to \sSets$ is a Quillen functor of two variables, where $\sSets = (\sSets, Qcat)$.
\end{df}
    The last condition of definition \ref{enrich-QCat} implies that if $A \in E$ is cofibrant and $X \in E$ is fibrant, then the simplicial set $[A,X]$ is a quasi-category.
If $E$ is a category with finite limits than so is the category $[\Delta^{op},E]$ of simplicial objects in E. The evaluation functor $ev_0 : [\Delta^{op},E] \to E$ defined by putting $ev_0(X) = X_0$ has a left adjoint $sk^0$ and a right adjoint $cosk^0$. If $A \in E$, then $sk^0(A)_n = A$ and $cosk^0(A)_n = A^{[n]} = A^{n+1}$ for every $n \ge 0$ (the simplicial object $sk^0(A)$ is the constant functor $cA : \Delta^{op} \to E$ with values $A$). The category $[\Delta^{op}, E]$ is simplicial. If $X, Y \in [\Delta^{op}, E]$ then we have
\[
[X,Y]n =Nat(X \circ p_n,Y \circ p_n)
\]
for every $n \ge 0$, where $p_n$ is the forgetful functor $\Delta/[n] \to \Delta$. If $A \in E$ and $cA := sk^0(A)$, then
\[
 [cA,X]_n = E(A,X_n)
 \]
for every $n \ge 0$. The simplicial category $[\Delta^{op},E]$ is tensored and cotensored by $\Delta$. By construction, if $X \in [\Delta^{op},E]$ and $K$ is a finite simplicial set, then
\begin{equation*}
(K \times X)_n = k_n \times X_n \ \ \ \ (X^K)_n = \int_{[k] \to [n]} X_k^{K_k}
\end{equation*}
The object $M_n(X) := (X \partial \Delta[n])_n$ is called the $n-th$ matching object of $X$. If $S(n)$ denotes the
poset of non-empty proper subsets of $[n]$ then we have
\[
M_n(X) =  \underset{S(n)} \liminj X \circ s(n)
\]
where $s(n) : S(n) \to \Delta$ is the canonical functor. From the inclusion $\partial \Delta[n] \subset \Delta[n]$ we obtain a map $X^{\Delta[n]} \to X^{\partial \Delta[n]}$ hence also a map $X_n \to M_n(X)$.

If $E$ is a model category, then a map $f : X \to Y$ in $[\Delta^{op}, E]$ is called a \emph{Reedy fibration} if the map $X_n \to Y_n \underset{M_n(Y )} \times  M_n(X)$ obtained from the square
\begin{equation*}
\xymatrix{
X_n \ar[d]_{f_n}  \ar[r] & M_n(X) \ar[d]^{M_n(f)} \\
Y_n \ar[r] & Mn(Y)
}
\end{equation*}
 is a fibration for every $n \ge 0$. There is then a model structure on the category $[\Delta^{op},E]$ called the \emph{Reedy model structure} whose fibrations are the Reedy fibrations and whose weak equivalences are the level-wise weak equivalences. A simplicial object $X : \Delta^{op} \to E$ is Reedy fibrant if and only if the canonical map $X_n \to M_n(X)$ is a fibration for every $n \ge 0$. The Reedy model structure is simplicial. If $X$ is Reedy fibrant and $A \in E$ then the simplicial set $E(A,X) := [cA,X]$ is a Kan complex.
 \begin{df}
Let $E$ be a model category. Then a simplicial object $Z : \Delta^{op} \to E$ is called a frame (see \cite{Hovey}) if the following two conditions are satisfied:
\begin{enumerate}
\item $Z$ is Reedy fibrant;
 \item $Z(f)$ is a weak equivalence for every map $f \in \Delta$.
\end{enumerate}
\end{df}
The frame $Z$ is cofibrant if the canonical map $sk^0Z_0 \to Z$ is a cofibration in the Reedy model structure. A \emph{coresolution} of an object $X \in E$ is a frame $Fr(X) : \Delta^{op} \to E$ equipped with a weak equivalence $X \to Fr(X)_0$. Every fibrant object $X \in E$ has a (cofibrant) coresolution $Fr(X) : \Delta^{op} \to E$ with $Fr(X)_0 = X$.
Let $E$ be a model category. If $A,X \in E$, then the homotopy mapping space $\HMapC{A}{X}{E}$ is defined to be the simplicial set
\begin{equation*}
\HMapC{A}{X}{E} = E(A^c,Fr(X))
\end{equation*}
where $A^c \to A$ is a cofibrant replacement of $A$ and $Fr(X)$ is a coresolution of $X$. The simplicial set $E(A^c,Fr(X))$ is a Kan complex and it is homotopy unique. If $E$is enriched over Kan complexes, if $A$ is cofibrant and $X$ is fibrant, then the simplicial set $\HMapC{A}{X}{E}$ is homotopy equivalent to the simplicial set $[A,X]$ (see \cite{Hirchhorn}).
\subsection{Function spaces for quasi-categories}
If $C$ is a category, we shall denote by $J(C)$ the sub-category of invertible arrows in $C$. The sub-category $J(C)$ is the largest sub-groupoid of $C$. More generally, if $X$ is a quasi-category, we shall denote by $J(X)$ the largest sub- Kan complex of $X$. By construction, we have a pullback square
\begin{equation*}
\xymatrix@C=20mm{
J(X) \ar[r] \ar[d] & X \ar[d]^h \\
J(\tau_1(X)) \ar[r] & \tau_1(X)
}
\end{equation*}
where $\tau_1(X)$ is the fundamental category of $X$ and $h$ is the canonical map. The function space $X^A$ is a quasi-category for any simplicial set $X$. We shall denote by $X^{(A)}$ the full sub-simplicial set of $X^A$ whose vertices are the maps $A \to X$ that factor through the inclusion $J(X) \subseteq X$. The simplicial set $X^{(\Delta[1])}$ is a path-space for $X$.

\begin{lem}
\label{cores-QCat}
 If $X$ is a quasi-category, then the simplicial object $P(X) \in [\Delta^{op},sSet]$ defined by putting $P(X)_n = X^{(\Delta[n])}$ for every $n \ge 0$ is a cofibrant coresolution of $X$.
\end{lem}
\begin{prop}
 If $X$ is a quasi-category and $A$ is a simplicial set, then
 \[ 
 \HMapC{A}{X}{\sSets} \simeq J(X^A).
 \]
 \end{prop}
 \begin{proof}
  Proof. By Lemma \ref{cores-QCat}, we have
  \[
\HMapC{A}{X}{\sSets}_n = \sSets(A,P(X)_n) = \sSets(A,X^{(\Delta[n])})
\]
But a map $f : A \to X^{\Delta[n]}$ factors through the inclusion $X^{(\Delta[n])} \subseteq X^{\Delta[n]}$ if and only if the
transposed map $f^t : \Delta[n] \to X^A$ factors through the inclusion $J(XA) \subseteq X^A$. Thus, $\sSets(A, X^{(\Delta[n])}) = \sSets(\Delta[n], J(X^A)) = J(X^A)_n$
and this shows that $\HMapC{A}{X}{\sSets} \simeq J(X^A)$. 
\end{proof}
\begin{prop}
\label{func-sSet-char}
 Let $E$ be a model category enriched over quasi-categories. If $A \in E$ is cofibrant and $X \in E$ is fibrant, then the function space $\HMapC{A}{X}{E}$ is equivalent to the Kan complex $J([A,X])$.
 \end{prop}
\begin{proof}
 The functor $[A, -]:$E$\to \sSets$ is a right Quillen functor with values in the model category $(\sSets,Qcat)$, since $A$ is cofibrant. It thus takes a coresolution $Fr(X)$ of $X \in E$ to a coresolution $[A,Fr(X)]$ of the quasi-category $[A,X]$. We have $\HMapC{1}{[A,X]}{\sSets} \simeq \sSets(1,P([A,X]))$, since the simplicial set $1$ is cofibrant. By Lemma \ref{cores-QCat}, the quasi-category $[A,X]$ has a cofibrant coresolution $P([A,X])$. We have $\HMapC{1}{[A,X]}{\sSets} \simeq \sSets(1,[A,Fr(X)])$, since the simplicial set $1$ is cofibrant. There exists a level-wise weak categorical equivalence $\phi : P([A,X]) \to [A,Fr(X)]$ such that the map $\phi(0)$ is the identity, since the coresolution $P([A,X])$ is cofibrant. Moreover, the map
 \[
  \sSets(1,\phi) : \sSets(1,P([A,X])) \to \sSets(1,[A,Fr(X)])
  \]
   is a weak homotopy equivalence. But we have
$\sSets(1,P([A,X])) = J([A,X])$ by lemma \ref{cores-QCat}. Moreover, $\sSets(1,[A,Fr(X)]) = E(A,Fr(X))$, since
\begin{multline*}
\sSets(1,[A,Fr(X)])_n = \sSets(1,[A,Fr(X)]_n) = \\
 \sSets(1,[A,Fr(X)_n]) = E(A, Fr(X)_n)
\end{multline*}
for every $n \ge 0$. 
\end{proof}
\subsection{Local objects}
Let $\Sigma$ be a set of maps in a model category $E$. An object $X \in E$ is said to be $\Sigma$-local if the map
\begin{equation*}
\HMapC{u}{X}{E} : \HMapC{A'} {X}{E} \to \HMapC{A}{ X}{E}
\end{equation*}
is a homotopy equivalence for every map $u:A \to A'$ in $\Sigma$. Notice that if an object $X$ is weakly equivalent to a $\Sigma$-local object, then $X$ is $\Sigma$-local. If the model category $E$ is simplicial (=enriched over Kan complexes) and $\Sigma$ is a set of maps between cofibrant objects, then a fibrant object $X \in E$ is $\Sigma$-local iff the map
$[u,X] : [A',X] \to [A,X]$ is a homotopy equivalence for every map $u : A \to A'$ in $\Sigma$.

\begin{lem}
\label{char-lo-QCat-en}
 Let $E$ be a model category enriched over quasi-categories. If $u : A \to B$ is a map between cofibrant objects, then the following conditions on a fibrant object $X \in E$ are equivalent
 \begin{enumerate}
\item the map $[u, X] : [B, X] \to [A, X]$ is a categorical equivalence;
\item the object $X$ is local with respect to the map $\Delta[n]\times u : \Delta[n] \times A \to \Delta[n] \times B$ for every
$n \ge 0$.
\end{enumerate}
\end{lem}
\begin{proof}
 (1 $\Rightarrow$ 2) The map $[u, X]^{\Delta[n]} : [B, X]^{\Delta[n]} \to [A, X]^{\Delta[n]}$ is a categorical equivalence for every $n \ge 0$, since the map $[u,X]$ is a categorical equivalence by the hypothesis. Hence the map $[\Delta[n] \times u, X]$ is a categorical equivalence, since $[\Delta[n] \times u, X ] = [u, X ]^{\Delta[n]}$. It follows that the map 
 $J([\Delta[n] \times u, X ])$ is a homotopy equivalence, since the functor $J : QCat \to Kan$ takes a
categorical equivalences to homotopy equivalences by \cite{AJ1}. But we have $\HMapC {\Delta[n] \times u}{X}{E} = J([\Delta[n] \times u, X])$ by Proposition \ref{func-sSet-char}, since $\Delta[n] \times u$ is a map between cofibrant objects. Hence the map $\HMapC{\Delta[n] \times u}{X}{E}$ is a homotopy equivalence for every $n \ge 0$. This shows that the object $X$ is local with respect to the map $\Delta[n] \times u$ for every $n \ge 0$. 

(1 $\Leftarrow$ 2) By Proposition \ref{func-sSet-char}, we have $\HMapC{\Delta[n] \times u}{X}{E} = J([\Delta[n]\times u,X])$ for every $n \ge 0$, since $\Delta[n] \times u$ is a map between cofibrant objects. Hence the map $J([\Delta[n] \times u, X])$ is a homotopy equivalence for every $n \ge 0$. But we have $[\Delta[n] \times u, X ] = [u, X]^{\Delta[n]}$. Hence the map $J([u, X ]^{\Delta[n]})$ is a homotopy equivalence for every $n \ge 0$. By Theorem 4.11 and Proposition 4.10 of \cite{JT2} a map between quasi-categories $f : U \to V$ is a categorical equivalence if and only if the map $J(f^{\Delta[n]}) : J(U^{\Delta[n]}) \to J(V^{ \Delta[n]})$ is a homotopy equivalence for every $n \ge 0$. This shows that the map $[u, X]$ is a categorical equivalence.
\end{proof}



%% file: EinMdlStrSp-Nor.tex
 \section[The strict JQ-model category structure on normalized $\gS$]{The strict JQ-model category of normalized $\gSs$}
\label{str-mdl-gcat}
A normalized $\gS$ is a functor $X:\gop \to \pSSets$ such that $X(0^+) = 1$.
The category of all (small) normalized $\gSs$ $\pGSC$ is the category whose objects are  normalized $\gSs$. This category is defined by the following equilizer diagram in $\Cat$:

\begin{equation*}
\xymatrix{
	\pGSC \ar[r] & [\gop; \pSSets] \ar[rd] \ar[rr]^{[0;\pSSets]} && [1; \pSSets] \\
	&& 1 \ar[ru]_0
}
\end{equation*}
 where $[0;\pSSets]$ is the functor which precomposes a functor in $[\gop; \pSSets]$ with the unique (pointed) functor $1 \to \gop$ whose image is $0^+ \in \gop$ and the upward diagonal functor $0$ maps the terminal category $1$ to the identity functor on the terminal simplicial sets. In this appendix we will describe a model category structure on the category
$\pGSC$ which is a version of the \emph{projective model category structure} for the category of
basepoint preserving functors.
\begin{df}
 A morphism $F:X \to Y$ of $\gSs$ is called
 \begin{enumerate}
 \item a \emph{strict JQ equivalence} of normalized $\gSs$  if it is degreewise weak equivalence in the Joyal model category structure on $\pSSets$ \emph{i.e.} $F(n^+):X(n^+) \to Y(n^+)$ is a weak categorical equivalence of (pointed) simplicial sets.
 
 \item a \emph{strict JQ fibration}  of normalized $\gSs$ if it is degreewise a fibration in the Joyal model category structure on $\pSSets$ \emph{i.e.} $F(n^+):X(n^+) \to Y(n^+)$ is an pseudo-fibration of (pointed) simplicial sets.
 
 \item a \emph{JQ-cofibration}  of normalized $\gSs$ if it has the left lifting property with respect to
 all morphisms which are both strict JQ weak equivalence and strict JQ fibrations of normalized $\gSs$.

  \end{enumerate}
 \end{df}
A map of $\gSs$ $F:X \to Y$ is a strict acyclic fibration of normalized $\gSs$ if and only if it has the right lifting property with respect to all maps in the set
 \begin{equation}
 \label{gen-cof}
 \I = \lbrace \TensP{\gn{n}}{\partial_0}{\pSSets},
 \pTensP{\gn{n}}{\partial_1}{\pSSets},  \pTensP{\gn{n}}{\partial_2}{\pSSets} \mid \forall n \in Ob(\N) \rbrace.
 \end{equation}
 We further observe that $F$ is a strict fibration if and only it has the right lifting property with respect to all maps in the set
 \begin{equation}
 \label{gen-acyc-cof}
 \J = \lbrace \pTensP{\gn{n}}{i_0}{\pSSets}, \pTensP{\gn{n}}{i_1}{\pSSets} \mid \forall n \in Ob(\N) \rbrace_{}.
 \end{equation}
 \begin{rem}
 \label{funct.-fact-sys-gCat}
 The category $\gS$ is a locally presentable category. The small object argument (for presentable categories), \cite[Proposition A.1.2.5]{JL}, implies that the sets $\I$ and $\J$ provide two functorial factorization systems on the category $\gS$. The first one factors each morphism in $\gS$ into a composite of
 a strict cofibration of $\gSs$ followed by a strict acyclic fibration of $\gSs$ and the second functorial factorization system factors each morphism in $\gS$ into a composite of
 a strict acyclic cofibration of $\gSs$ followed by a strict fibration of $\gSs$. 
\end{rem}
 The main aim of this subsection is to construct a model category structure on the category
 of all $\gSs$ $\gS$ whose three classes of morphisms are the ones defined
 above. We will refer to this model structure as the \emph{strict model category structure} on $\gS$
 and will refer to the model category as the \emph{strict model category} of $\gSs$.
 \begin{thm}
 \label{str-mdl-cat-gCat}
 Strict JQ equivalences, strict JQ fibrations and JQ-cofibrations of $\gSs$ provide the category $\pGSC$ with a combinatorial model category structure.
 \end{thm}
 \begin{proof}
 The category of all functors from $\gop$ to $\pSSets$, namely $\CatHom{\gop}{\pSSets}{}$ has a
 model category structure, called the \emph{projective model category structure},
 in which a map is a weak equivalence (resp. fibration) if and only if it is a weak equivalence
 (resp. fibration) degreewise, see \cite[Prop. A.3.3.2]{JL} for a proof. The category 
 $\gS = \pHomCat{\gop}{\pSSets}$ is a subcategory of $\CatHom{\gop}{\pSSets}{}$
 this implies that the axioms $CM(2), CM(3)$ and $CM(4)$, see \cite{Quillen}, \cite[Chap. 2]{GJ} are satisfied by $\gS$ because they are satisfied by the projective model category $\CatHom{\gop}{\pSSets}{}$.
 Finally, $CM(5)$ follows from remark \ref{funct.-fact-sys-gCat} above. The category $\gS$
 is locally presentable. The sets $\I$ and $\J$ defined above form the sets of generating cofibrations
 and generating acyclic cofibrations respectively of the strict model category structure.
 \end{proof}
\begin{nota}
	We will refer to the above model category as the \emph{strict JQ model category of normalized $\gSs$} and we denote it by $\pGSCStr$.
	\end{nota}
We recall that the \emph{smash product} of two (pointed) simplicial sets $(X, x)$ and $(Y, y)$, where the simplicial maps $x:1 \to X$ and $y:1 \to Y$ specify the respective basepoints, is defined by the following pushout square:
\begin{equation}
\label{smash-sSets}
\xymatrix{
X \vee Y \ar[d] \ar[r] & X \times Y \ar[d] \\
1 \ar[r] & X \wedge Y
}
\end{equation}
where the top horizontal arrow is the canonical map between the coproduct and product of the two (pointed) simplicial sets.
 To any pair of objects $(X, C) \in Ob(\pGSC) \times Ob(\pSSets)$ we can assign a $\gS$ $\TensP{X}{C}{\pSSets}$
 which is defined in degree $n$ as follows:
 \[
 (\TensP{X}{C}{\pSSets})(n^+) :=  \pTensP{X(n^+)}{C}{},
 \]
 where the pointed category on the right is the smash product of (pointed) simplicial sets, see \eqref{smash-sSets}. This assignment
 is functorial in both variables and therefore we have a bifunctor
 \[
 \TensP{-}{-}{\pSSets}:\pGSC \times \pSSets \to \pGSC.
 \]
 Next,we define a couple of function objects for the category $\gS$.
 The first function object enriches the category $\pGSC$ over
 $\pSSets$ \emph{i.e.} there is a bifunctor
 \[
 \MapC{-}{-}{\pGSC}:\gS^{op} \times \pGSC \to \pSSets
 \]
 which assigns to any pair of objects $(X, Y) \in Ob(\pGSC) \times Ob(\pGSC)$, a pointed simplicial set
 $\MapC{X}{Y}{\pGSC}$ which is defined in degree zero as follows:
 \[
 \MapC{X}{Y}{\pGSC}_0 := \pGSC(X, Y).
 \]
 The mapping simplicial set is deined in degree $n$ as follows:
 \[
 \MapC{X}{Y}{\pGSC}_n := \pGSC(X \wedge \Delta[n]^+, Y)
 \]
 For any $\gS$ $X$, the functor $\TensP{X}{-}{\pSSets}:\pSSets \to \pGSC$ is
 left adjoint to the functor $\MapC{X}{-}{\pGSC}:\pGSC \to \pSSets$. The counit of this adjunction
 is the evaluation map $ev:\TensP{X}{\MapC{X}{Y}{\pGSC}}{\pSSets} \to Y$
 and the unit is the obvious functor $C \to \MapC{X}{\TensP{X}{C}{\pSSets}}{\pGSC}$, where $Y$ is a normalized $\gS$ and $C$ is a pointed simplicial set.
 
 The mapping object $\MapC{X}{Y}{\pGSC}$ is a (pointed) simplicial set whose basepoint is the composite map $X \to \gn{0} \to Y$, where $\gn{0}$ is the zero object in $\pGSC$. Let $U(\MapC{X}{Y}{\pGSC})$ denote the simplicial set obtained by forgetting the basepoint of $\MapC{X}{Y}{\pGSC}$. We also recall the forgetful functor $U$ which forgets the normalization of a $\gS$, see \eqref{Forget-fun-GS}.
 \begin{lem}
 	Let $X$ and $Y$ be two normalized $\gSs$. The mapping simplicial set $U(\MapC{X}{Y}{\pGSC})$ is an equilizer of the following 
 	diagram:
 	\begin{equation*}
 	\xymatrix{
 	\MapC{U(X)}{U(Y)}{\gSC} \ar[rr] \ar[rd] && \MapC{\gn{0}}{U(Y)}{\gSC} \\
 	& 1 \ar[ru]_0
 }
 	\end{equation*}
 	\end{lem}
 \begin{proof}
 	Each normalized $\gS$ $X$ uniquely determines a morphism $0_X:\gn{0} \to U(X)$. It is sufficient to observe that a morphism $f:U(X) \to U(Y)$ lies in the image of the forgetful functor $U$ if and only if the following diagram commutes:
 	\[
 	\xymatrix{
 	& \gn{0} \ar[ld]_{0_X} \ar[rd]^{0_Y} \\
    U(X) \ar[rr]_f && U(Y)
    }
 	\]
 	\end{proof}
 \begin{coro}
 	\label{nor-map-sp}
 	For each pair of normalized $\gSs$ $X$ and $Y$ we have the following canonical isomorphism of mapping simplicial sets
 	\[
 	U(\MapC{X}{Y}{\pGSC}) \cong \MapC{U(X)}{U(Y)}{\gSC}.
 	\]
 	\end{coro}
 \begin{proof}
 	It is sufficient to observe that for any normalized $\gS$ $Y$, the Yoneda's lemma tells us that the mapping simplicial set $\MapC{\gn{0}}{U(Y)}{\gSC} \cong 1$.
 	\end{proof}
 
 To each pair of objects $(C, X) \in Ob(\pSSets) \times Ob(\gS)$ we can assign a $\gS$ $X^C$
 which is defined in degree $n$ as follows:
 \[
 (X^C_\bullet)(n^+) := X(n^+)^C_\bullet \,
 \]
 where the (pointed) simplicial set on the right is is defined by the following equilizer diagram:
 \begin{equation*}
 \xymatrix{
 	X(n^+)^C_\bullet  \ar[r] & [C; X(n^+)] \ar[rd] \ar[rr]^{[c;X(n^+)]} && [1; X(n^+)] \\
 	&& 1 \ar[ru]_0
 }
 \end{equation*}
 where $c:1 \to C$ is the basepoint map.
 This assignment
 is functorial in both variable and therefore we have a bifunctor
 \[
 \bHom{-}{-}{\pGSC}:\pSSets^{op} \times \pGSC \to \pGSC.
 \]
 For any $\gS$ $X$, the functor $\bHom{-}{X}{\gS}:\pSSets \to \gS^{op}$ is
 left adjoint to the functor $\MapC{-}{X}{\gS}:\gS^{op} \to \pSSets$. 
 
The following proposition summarizes the above discussion.
\begin{prop}
\label{two-var-adj-qcat-nor-gcat}
There is an adjunction of two variables
\begin{equation}
\label{two-var-adj-nor-gcat}
(\pTensP{-}{-}{\pSSets}, \bHom{-}{-}{\pGSC}, \MapC{-}{-}{\pGSC}) : \pGSC \times \pSSets  \to \pGSC.
\end{equation}

\end{prop}

 \begin{thm}
  \label{enrich-p-CAT-Q}
  The strict model category of normalized $\gSs$, $\pGSC$, is a $\pSSets$-enriched model category.
 \end{thm}
 \begin{proof}
  We will show that the adjunction of two variables \eqref{two-var-adj-gcat}
  is a Quillen adjunction for the strict model category structure
 on $\gS$ and the natural model category structure on $\pSSets$.
  In order to do so, we will verify condition
 (2) of Lemma \ref{Q-bifunctor-char}. Let $g:C \to D$ be a cofibration
 in $\pSSets$ and let $p:Y \to Z$ be a strict fibration of $\gSs$,
 we have to show that the induced map
 \[
  \bhom^{\Box}_{\gS}(g, p):\bHom{X}{Y}{\gS} \to \bHom{D}{Z}{\gS}
  \underset{\bHom{C}{Z}{\gS}} \times \bHom{C}{Y}{\gS}
 \]
 is a fibration in $\pSSets$ which is acyclic if either of $g$ or $p$ is
 acyclic. It would be sufficient to check that the above morphism is degreewise
 a fibration in $\pSSets$, i.e. for all $n^+ \in \gop$, the morphism
 \begin{equation*}
  \bhom^{\Box}_{\gS}(g, p)(n^+): \pHomCat{D}{Y(n^+)} \ \to
  \pHomCat{D}{Z(n^+)}  \
  \underset{ \pHomCat{C}{Z(n^+)} }
  \times  \pHomCat{C}{Y(n^+)} ,
 \end{equation*}
 is a fibration in $\pSSets$. This follows from the observations that the functor $p(n^+):Y(n^+) \to Z(n^+)$
 is a fibration in $\pSSets$ and the natural model category 
 $\pSSets$ is a $\pSSets$-enriched model category whose enrichment is provided by the bifunctor $\pHomCat{-}{-}$.
 \end{proof}
 The adjunction $-^+:\sSets \rightleftharpoons \pSSets:U $ provides us with an enrichment
 of the strict model category of $\gSs$, $\gS$, over the natural model category of
 all (small) categories $\Cat$.
 \begin{coro}
\label{enrich-CAT-Q}
The strict model category of normalized $\gSs$, $\pGSC$, is a $\sSetsQ$-enriched model category.
 \end{coro}

\input{EinMdlStrSp-Nor-SMStr}

%% file: EinMdlStrSp-Nor-SMStr.tex
\subsection[The $JQ$-model category of normalized $\gSs$]{The $JQ$-model category of normalized $\gSs$}
\label{JQ-su-ccm-qcat}

The objective of this subsection is to construct a new model
category structure on the category $\pGSC$. This new model
category is obtained by localizing the strict JQ-model category of normalized $\gSs$
(see section \ref{str-mdl-gcat}) and we we refer to it as the \emph{JQ-model category of normalized $\gSs$}.  We go on further to show that this
new model category is symmetric monoidal closed with respect to
the smash product which is a categorical version of the smash product
constructed in \cite{lydakis}.
\begin{nota}
	\label{nor-pointed-GS}
	We denote by $1/\gSC$ the overcategory whose objects are maps in $\gSC$
	having domain the terminal $\gS$ $1$. We denote by $(1/\gSC)_\bullet$ the subcategory of $1/\gSC$ whose objects are those maps $1 \to X$ in $\gSC$ 
	whose codomain $\gS$ satisfies the following normalization condition:
	\[
	X(0^+) = \ast. 
	\]
	 
	\end{nota}
\begin{rem}
	\label{isom-nor-pointed-obj}
	We observe that the category of normalized pointed objects $(1/\gSC)_\bullet$ is isomorphic to the category of normalized $\gSs$ $\pGSC$.
	\end{rem}

 We want to construct a left Bousfield localization of
 the strict model category of $\gSs$. For each pair $k^+, l^+ \in \gop$,
 we have the obvious \emph{projection maps} in $\gop$
 \[
  \delta^{k+l}_k:(k+l)^+ \to k^+ \ \ \ \ and \ \ \ \ \delta^{k+l}_l:(k+l)^+ \to l^+.
 \]
 The following two inclusion maps between representable $\gSs$
 \[
 \gop(\delta^{k+l}_k,-):\Gamma^{k} \to \Gamma^{k+l} \ \ \ \ and \ \ \ \ 
 \gop(\delta^{k+l}_l,-):\Gamma^{l} \to \Gamma^{k+l} 
 \]
 induce a pair of maps of $\gSs$ on the coproduct which we denote as follows:
 \[
  h_k^l:\Gamma^l \vee \Gamma^l \to \Gamma^{l+k}.
 \]
 
 We now define a set of maps $\E_\infty\S_\bullet$ in $\pGSC$:
 \begin{equation*}
  \E_\infty\S_\bullet := \lbrace h_k^l:\Gamma^l \vee \Gamma^l \to \Gamma^{l+k}:
  l, k \in \mathbb{Z}^+ \rbrace
 \end{equation*}
 Next we define the set of arrows in $\pGSC$ with respect to which we will localize the strict $JQ$-model category of normalized $\gSs$:
 \begin{equation*}
 \Delta \times \E_\infty\S_\bullet := \lbrace \TensP{\Delta[n]^+}{h_k^l}{\pSSets} :h_k^l \in \Delta \times \E_\infty\S_\bullet \rbrace
 \end{equation*}
  \begin{df}
  	We call a $\gS$ $X$ a $(\Delta \times \E_\infty\S_\bullet)$-\emph{local object}
  	if it is a fibrant object in the strict $JQ$-model category of normalized $\gSs$ and for each map $\TensP{\Delta[n]^+}{h_k^l}{\pSSets} \in \Delta \times \E_\infty\S_\bullet$, the induced simplicial map
  	\begin{multline*}
  	\HMapC{\TensP{\Delta[n]^+}{h_k^l}{\pSSets}}{X}{\pGSC}: \HMapC{\TensP{\Delta[n]^+}{\gn{k+l}}{\pSSets}}{X}{\pGSC} \to \\
  	\HMapC{\TensP{\Delta[n]^+}{(\gn{k} \vee \gn{l})}{\pSSets}}{X}{\pGSC},
  	\end{multline*}
  	is a homotopy equivalence of simplicial sets for all $n \ge 0$ where $\HMapC{-}{-}{\pGSC}$ is the simplicial function complexes associated with the strict model category $\pGSC$, see \cite{DK80}, \cite{DK1980} and \cite{DK3}.
  \end{df}
  Corollary \ref{enrich-CAT-Q} above and appendix \ref{Cat-Local} tell us that a model for $\HMapC{X}{Y}{\pGSC}$ is the Kan complex
  $J(\MapC{X}{Y}{\pGSC})$ which is the maximal kan complex contained in the quasicategory $\MapC{X}{Y}{\pGSC}$.

  The following proposition gives a characterization of
  $\Delta \times \E_\infty\S_\bullet$-local objects
  \begin{prop}
  	\label{char-CCMC}
  	\begin{sloppypar}
  		A normalized $\gS$ $X$ is a $(\Delta \times \E_\infty\S_\bullet)$-local object  if and only if it satisfies the Segal condition namely the functor
  	\end{sloppypar}
  	\begin{equation*}
  	(X(\partition{k+l}{k}), X(\partition{k+l}{l})):X((k+l)^+) \to X(k^+) \times X(l^+)
  	\end{equation*}
  	is an equivalence of (pointed) quasi-categories for all $k^+, l^+ \in \Ob(\gop)$.
  \end{prop}
  \begin{proof}
  	We begin the proof by observing that each element of the set $\Delta \times \E_\infty\S_\bullet$ is a map of $\gSs$ between cofibrant $\gSs$. Theorem \ref{char-lo-QCat-en} implies that $X$ is a $(\Delta \times \E_\infty\S_\bullet)$-local object if and only if the following map of simplicial sets
  	\begin{equation*}
  	\MapC{h^k_l}{X}{\pGSC}:\MapC{\gn{k+l}}{X}{\pGSC} \to \MapC{\gn{k} \vee \gn{l}}{X}{\pGSC}
  	\end{equation*}
  	is an equivalence of quasi-categories.
  	
  	We observe that we have the following commutative square in $\sSetsQ$
  	\begin{equation*}
  	\xymatrix@C=24mm{
  		\MapC{\gn{k+l}}{X}{\pGSC}  \ar[d]_{\cong}\ar[r]^{ \MapC{h^k_l}{X}{\pGSC}} & \MapC{\gn{k} \vee \gn{l}}{X}{\pGSC}   \ar[d]^{\cong} \\
  		X((k+l)^+) \ar[r]_{(X(\partition{k+l}{k}), X(\partition{k+l}{l}))} & X(k^+) \times X(l^+)
  	}
  	\end{equation*}
  	By the two out of three property of weak equivalences in a model category the  simplicial map $(X(\partition{k+l}{k}), X(\partition{k+l}{l}))$ is an equivalence of  quasi-categories if and only if the map $ \MapC{h^k_l}{X}{\pGSC}$ is an equivalence of quasi-categories.
  \end{proof}
  
  \begin{df}
  	\label{CCMC}
  	We will refer to a $(\Delta \times \E_\infty\S_\bullet)$-local object as a \emph{normalized coherently commutative monoidal quasi-category}.
  \end{df}

  \begin{df}
 	A morphism of normalized $\gSs$ $F:X \to Y$ is a $(\Delta \times \E_\infty\S_\bullet)$-local equivalence if for each normalized coherently commutative monoidal quasi-category $Z$ the following simplicial map
 	\[
 	\HMapC{F}{Z}{\pGSC}:\HMapC{Y}{Z}{\pGSC} \to \HMapC{X}{Z}{\pGSC}
 	\]
 	is a homotopy equivalence of simplicial sets. We may sometimes refer to a $(\Delta \times \E_\infty\S_\bullet)$-local equivalence as an equivalence of normalized coherently commutative monoidal quasi-categories.
 \end{df}
An argument similar to the proof of proposition \ref{char-CCME} proves the following proposition:
 \begin{prop}
 	\label{char-CCME-nor}
 	\begin{sloppypar}
 		A morphism between two $JQ$-cofibrant normalized $\gSs$ $F:X \to Y$ is an $(\Delta \times \E_\infty\S_\bullet)$-local equivalence if and only if the simplicial map  \end{sloppypar}
 	\begin{equation*}
 	\MapC{F}{Z}{\pGSC}:\MapC{Y}{Z}{\pGSC} \to \MapC{X}{Z}{\pGSC}
 	\end{equation*}
 	is an equivalence of quasi-categories for each normalized coherently commutative monoidal quasi-category $Z$.
 \end{prop}
 The main objective of the current subsection is to construct a
 new model category structure on the category of normalized $\gSs$ $\pGSC$
 by localizing the strict $JQ$-model category of normalized $\gSs$ with respect to morphisms in the set $\Delta \times \E_\infty\S_\bullet$. The desired model structure follows from theorem \ref{local-tool}
 
\begin{thm}
 \label{main-result-nor}
 There is a closed, left proper, combinatorial model category structure on
 the category of normalized $\gSs$, $\pGSC$, in which
 \begin{enumerate}
 \item The class of cofibrations is the same as the class of
 $JQ$-cofibrations of normalized $\gSs$.
 \item The weak equivalences are equivalences of normalized coherently commutative monoidal quasi-categories.
 \end{enumerate}
 An object is fibrant in this model category if and only if it is a
  normalized coherently commutative monoidal quasi-category. Further, this model category structure makes
  $\pGSC$ a closed symmetric monoidal model category under
 the smash product.
 
 \end{thm}
 \begin{proof}
 The strict $JQ$-model category of normalized $\gSs$ is a combinatorial
 model category therefore the existence of the model structure
 follows from theorem \ref{local-tool}. The statement
 characterizing fibrant objects also follows from theorem \ref{local-tool}.
 An argument similar to the proof of theorem \ref{SM-closed-CCMC} using the enrichment of the strict $JQ$-model category of normalized $\gSs$ over the $\pSSetsQ$ established in proposition \ref{two-var-adj-cat-gcat} shows that the localized model category has a symmetric monoidal closed model category structure under the smash product. 
 \end{proof}
 \begin{nota}
 The model category constructed in theorem \ref{main-result-nor} will
 be referred to either as \emph{the $JQ$-model category of normalized $\gSs$}
 or as the model category of normalized coherently commutative monoidal quasi-categories.
 \end{nota}